\documentclass[11pt]{amsart}
\usepackage[margin=1in]{geometry}
\usepackage{latexsym}

\usepackage{amsfonts}
\usepackage{amsmath}
\usepackage{amssymb}
\usepackage{amsthm}
\usepackage{enumerate}
\usepackage[hang,flushmargin]{footmisc}
\usepackage{caption}
\usepackage{tabu}
\usepackage{mathrsfs}
\usepackage{amsaddr}
\usepackage{subfig}
\usepackage{mathtools}

\usepackage[dvipsnames]{xcolor}

\usepackage{yhmath}
\usepackage{graphicx}
\usepackage{epstopdf}
\usepackage{epsfig}
\usepackage{caption}

\DeclareFontFamily{U}{mathx}{\hyphenchar\font45}
\DeclareFontShape{U}{mathx}{m}{n}{
      <5> <6> <7> <8> <9> <10>
      <10.95> <12> <14.4> <17.28> <20.74> <24.88>
      mathx10
      }{}
\DeclareSymbolFont{mathx}{U}{mathx}{m}{n}
\DeclareFontSubstitution{U}{mathx}{m}{n}
\DeclareMathAccent{\widecheck}{0}{mathx}{"71}
 
\usepackage{bm} 
\usepackage[noadjust]{cite}

\makeatletter

\newtheorem{theorem}{Theorem}[section]
\newtheorem{lemma}[theorem]{Lemma}
\newtheorem{proposition}[theorem]{Proposition}
\newtheorem{corollary}[theorem]{Corollary}

\newtheorem{conjecture}[theorem]{Conjecture}

\newtheorem{problem}[theorem]{Problem}

\theoremstyle{definition}
\newtheorem{definition}[theorem]{Definition}
\newtheorem{remarkx}[theorem]{Remark}
\newenvironment{remark}
  {\pushQED{\qed}\remarkx}
  {\popQED\endremarkx}
\newtheorem{examplex}[theorem]{Example}
\newenvironment{example}
  {\pushQED{\qed}\examplex}
  {\popQED\endexamplex}

\DeclareMathOperator{\Pop}{\mathsf{Pop}}
\DeclareMathOperator{\syl}{syl}
\DeclareMathOperator{\Vecc}{Vec}
\DeclareMathOperator{\Cov}{Cov}
\DeclareMathOperator{\Av}{Av}
\DeclareMathOperator{\rev}{rev}
\DeclareMathOperator{\tl}{tl}
\DeclareMathOperator{\Tam}{Tam}
\DeclareMathOperator{\N}{N}
\DeclareMathOperator{\E}{E}

\newcommand{\dfn}[1]{\textcolor{blue}{\emph{#1}}}

\begin{document}
\title{Meeting Covered Elements in $\nu$-Tamari Lattices}
\author{Colin Defant}
\address{Princeton University \\ Department of Mathematics \\ Princeton, NJ 08544}
\email{cdefant@princeton.edu}

\begin{abstract}
For each meet-semilattice $M$, we define an operator $\Pop_M:M\to M$ by \[\Pop_M(x)=\bigwedge(\{y\in M:y\lessdot x\}\cup\{x\}).\] When $M$ is the right weak order on a symmetric group, $\Pop_M$ is the pop-stack-sorting map. We prove some general properties of these operators, including a theorem that describes how they interact with certain lattice congruences. We then specialize our attention to the dynamics of $\Pop_{\Tam(\nu)}$, where $\Tam(\nu)$ is the $\nu$-Tamari lattice. We determine the maximum size of a forward orbit of $\Pop_{\Tam(\nu)}$. When $\Tam(\nu)$ is the $n^\text{th}$ $m$-Tamari lattice, this maximum forward orbit size is $m+n-1$; in this case, we prove that the number of forward orbits of size $m+n-1$ is \[\frac{1}{n-1}\binom{(m+1)(n-2)+m-1}{n-2}.\] Motivated by the recent investigation of the pop-stack-sorting map, we define a lattice path $\mu\in\Tam(\nu)$ to be \emph{$t$-$\Pop$-sortable} if $\Pop_{\Tam(\nu)}^t(\mu)=\nu$. We enumerate $1$-$\Pop$-sortable lattice paths in $\Tam(\nu)$ for arbitrary $\nu$. We also give a recursive method to generate $2$-$\Pop$-sortable lattice paths in $\Tam(\nu)$ for arbitrary $\nu$; this allows us to enumerate $2$-$\Pop$-sortable lattice paths in a large variety of $\nu$-Tamari lattices that includes the $m$-Tamari lattices. 
\end{abstract}

\maketitle

\bigskip

\section{Introduction}\label{Sec:Intro} 

\subsection{Pop-Stack-Sorting}
Suppose $X$ is a set of combinatorial objects and $f:X\to X$ is a function defined in some combinatorial manner. We can view $f$ as a dynamical system and define the \dfn{forward orbit} of an element $x\in X$ under $f$ to be the set $O_f(x)=\{x,f(x),f^2(x),\ldots\}$. When $f$ is invertible, a standard line of questioning involves the orbit structure of $f$. For example, it is natural to ask for the maximum size of a forward orbit. On the other hand, the aim of \emph{noninvertible combinatorial dynamics} is to study this setup when $f$ is not invertible, often focusing on the transient (i.e., non-periodic) points. As in the invertible case, it is natural to ask for the maximum size of a forward orbit; this question was explored for some specific noninvertible combinatorial dynamical systems in \cite{AlbertVatter, DefantCoxeterPop, PromotionSorting, DefantZheng, Bentz, Igusa, Etienne, Griggs, Ungar, West}. Another typical line of questions, especially in the case of \emph{sorting operators}, asks for the characterization and/or enumeration of elements of $X$ that require at most $t$ iterations of $f$ to reach a fixed point (see \cite{DefantCounting, BonaSurvey, Zeilberger, Goulden, ClaessonPop, Pudwell, PromotionSorting, Elder, Albert} and the references therein). 

A \dfn{meet-semilattice} is a poset $M$ such that any two elements $x,y\in M$ have a greatest lower bound, which is called their \dfn{meet} and denoted by $x\wedge y$. A \dfn{lattice} is a meet-semilattice $M$ such that any two elements $x,y\in M$ have a least upper bound, which is called their \dfn{join} and denoted by $x\vee y$. We write $\bigwedge A$ for the meet of a set $A\subseteq M$. Given elements $x$ and $y$ in a poset $P$, we write $y\lessdot x$ to mean that $y$ is covered by $x$ (i.e., $y<x$, and there does not exist $z\in P$ with $y<z<x$). Every meet-semilattice $M$ considered in this article is assumed to have the property that the set $\{y\in M:y\lessdot x\}$ exists and has a meet for every $x\in M$. The main definition in this article introduces, for each meet-semilattice $M$, a natural noninvertible dynamical system on $M$. 

\begin{definition}\label{DefMeet1}
Let $M$ be a meet-semilattice. Define the \dfn{semilattice pop-stack-sorting operator} $\Pop_M:M\to M$ by \[\Pop_M(x)=\bigwedge(\{y\in M:y\lessdot x\}\cup\{x\}).\]
\end{definition}

The only reason for using the set $\{y\in M:y\lessdot x\}\cup\{x\}$ instead of $\{y\in M:y\lessdot x\}$ in Definition~\ref{DefMeet1} is to ensure that if $M$ has a minimal element $\widehat 0$, then $\Pop_M(\widehat 0)=\widehat 0$. 

The motivation behind Definition~\ref{DefMeet1} comes from the \emph{pop-stack-sorting map}, which is a map defined on the symmetric group $S_n$ as a deterministic variant of a pop-stack-sorting machine that was originally introduced by Avis and Newborn in \cite{Avis} (see Section~\ref{SecCongruences} for its definition). This map originally appeared under a different guise in a paper of Ungar's \cite{Ungar} about directions determined by points in the plane, and it has received a great deal of attention from enumerative combinatorialists in the past few years \cite{AlbertVatter, Asinowski, Asinowski2, DefantCoxeterPop, Elder, ClaessonPop, ClaessonPop2, Pudwell, Ungar}. In the recent article \cite{DefantCoxeterPop}, the current author viewed the pop-stack-sorting map on $S_n$ from a different perspective, observing that it is equal to $\Pop_M$ when $M$ is the right weak order on $S_n$. This motivated the study of $\Pop_M$ when $M$ is the right weak order on an arbitrary Coxeter group.

Recent work has focused on \emph{$t$-pop-stack-sortable} permutations, which are permutations that require at most $t$ iterations of the pop-stack-sorting map to reach the identity \cite{AlbertVatter, Elder, ClaessonPop, Pudwell}. There has also been a great deal of work devoted to \emph{$t$-stack-sortable} permutations in the study of West's stack-sorting map, especially for $t\in\{1,2,3\}$ (see \cite{DefantCounting, DefantElvey, BonaSurvey, Zeilberger, Goulden, Albert} and the references therein). This work is our motivation for the following definition. Suppose $M$ is a meet-semilattice with a minimal element $\widehat 0$. We say an element $x\in M$ is \dfn{$t$-$\Pop$-sortable} if $\Pop_M^t(x)=\widehat 0$. 

The \emph{Tamari lattices} $\Tam_n$ were originally introduced by Tamari in \cite{Tamari} and have now become fundamental in algebraic combinatorics \cite{Muller}. It is common to define Tamari lattices on Dyck paths, as we do in Section~\ref{SubsecLattice}. There are now several generalizations of Tamari lattices. For example, Bergeron and Pr\'eville-Ratelle \cite{Bergeron} introduced \emph{$m$-Tamari lattices} in order to give (still open) conjectural combinatorial interpretations of the dimensions of certain spaces arising from the study of trivariate diagonal harmonics. These lattices are defined on \dfn{$m$-ballot paths}, which are lattice paths that start at $(0,0)$, end at $(mn,n)$, use only unit north and east steps, and stay weakly above the line $y=x/m$. The $m$-Tamari lattices have now been studied extensively (see \cite{BousquetRep, BousquetIntervals, Chatel} and the references therein). In \cite{PrevilleViennot}, Pr\'eville-Ratelle and Viennot introduced the $\nu$-Tamari lattice $\Tam(\nu)$, which is defined on the set of lattice paths lying weakly above a fixed lattice path $\nu$. These lattices are far-reaching generalizations of $m$-Tamari lattices; they have now been investigated further in \cite{Ceballos, Ceballos2, vonBell, vonBell2, FangPreville}. We define Tamari, $m$-Tamari, and $\nu$-Tamari lattices in Section~\ref{SubsecLattice}.   

\subsection{Other Incarnations of Pop-Stack-Sorting}

As mentioned above, the pop-stack-sorting map can be viewed as the semilattice pop-stack-sorting operator on the right weak order of $S_n$. Let us briefly mention some other places where pop-stack-sorting arises naturally. 

When $M$ is the lattice of regions of a real simplicial hyperplane arrangement, Reading considered $\Pop_M$ in relation to his \emph{shard intersection order} (see \cite[Section 9-7]{ReadingBook} and \cite{ReadingShard}). This notion has received further attention for congruence uniform lattices in \cite{Muhle1} and for arbitrary lattices in \cite{Muhle2}. 

Let $\Phi^+(A_{n-1})$ be the root poset of type $A_{n-1}$. There is a natural bijection between the set of order ideals of $\Phi^+(A_{n-1})$ and the set of Dyck paths of semilength $n$. When ordered by inclusion, these order ideals form a distributive lattice $\mathcal J(\Phi^+(A_{n-1}))$. The article \cite{Sapounakis} defines a certain \emph{filling} operator on Dyck paths and studies the image of this operator along with some of its dynamical properties. It is easy to show that the filling operator on Dyck paths corresponds exactly to the pop-stack-sorting operator on the dual of $\mathcal J(\Phi^+(A_{n-1}))$.  

There is also a natural appearance of semilattice pop-stack-sorting operators that deals with chip-firing. Since we only plan to mention this connection in passing, we refer the reader to \cite[Chapter~2]{Klivans} for the relevant terminology. Suppose ${\bf c}$ is a chip configuration for a finite graph $G$ such that the chip-firing process starting from ${\bf c}$ stabilizes. Let $M_{\bf c}$ be the set of chip configurations for $G$ that are reachable from ${\bf c}$. Define the partial order $\leq$ on $M_{\bf c}$ by saying ${\bf d}'\leq{\bf d}$ if ${\bf d}'$ is reachable from ${\bf d}$. Latapy and Phan \cite{Latapy} proved that $M_{\bf c}$ is a lattice. For each ${\bf d}\in M_{\bf c}$, the chip configuration $\Pop_{M_{\bf c}}({\bf d})$ is obtained from ${\bf d}$ by simultaneously firing (also called \emph{cluster firing}) all vertices that are ready to fire in ${\bf c}$. This setup is very similar to \emph{parallel chip-firing}, which was initiated in \cite{Bitar} and has received significant attention thereafter. However, parallel chip-firing is different from our setup because it allows vertices to have negative numbers of chips (so there are no stable configurations).

\subsection{Outline}

In Section~\ref{SecPopTrivial}, we define a meet-semilattice $M$ with a minimal element $\widehat 0$ to be \emph{$\Pop$-trivial} if $\Pop_M(x)=\widehat 0$ for all $x\in M$. We show that geometric lattices are $\Pop$-trivial and ask for a characterization of finite $\Pop$-trivial lattices. 

Suppose $\equiv$ is a lattice congruence on a locally finite lattice $M$. We write $\pi_\downarrow(x)$ for the minimal element of the congruence class of $\equiv$ containing $x$. In Section~\ref{SecCongruences}, we prove that if the set $M'=\{\pi_\downarrow(x):x\in M\}$ is a sublattice of $M$, then $\Pop_{M'}(x)=\pi_\downarrow(\Pop_M(x))$ for all $x\in M'$. When $\equiv$ is the sylvester congruence on the right weak order of the symmetric group $S_n$, the set $M'$ is isomorphic to the Tamari lattice $\Tam_n$. Thus, we obtain a combinatorial description of the action of $\Pop_{\Tam_n}$ in terms of the pop-stack-sorting map on permutations. In fact, we will recast this description as a combination of the pop-stack-sorting map and West's stack-sorting map. 

Section~\ref{SecTamari} is devoted to studying the semilattice pop-stack-sorting operators on $\nu$-Tamari lattices. We obtain several general results that apply when $\nu$ is an arbitrary lattice path, but all of the theorems we prove are new even for Tamari lattices. Our first main theorem in Section~\ref{SecTamari} (Theorem~\ref{ThmMeet2}) provides a formula for the maximum size of a forward orbit of the dynamical system $\Pop_{\Tam(\nu)}:\Tam(\nu)\to\Tam(\nu)$. When we specialize to the $m$-Tamari lattice $\Tam_n(m)$, the formula simplifies to $m+n-1$. We then characterize the lattice paths $\mu\in\Tam_n(m)$ such that $\left\lvert O_{\Pop_{\Tam_n(m)}}(\mu)\right\rvert=m+n-1$, and we use generating trees to prove that the number of such lattice paths is \[\frac{1}{n-1}\binom{(m+1)(n-2)+m-1}{n-2}.\] This is also the number of \emph{primitive} $m$-ballot paths in $\Tam_{n-1}(m)$, though the precise connection between the maximum-size forward orbits and the primitive $m$-ballot paths remains mysterious at this time. 

The remainder of Section~\ref{SecTamari} focuses on $1$-$\Pop$-sortable and $2$-$\Pop$-sortable elements of $\nu$-Tamari lattices. We prove that if $\nu=\text{E}^{\gamma_0}\text{NE}^{\gamma_1}\cdots\text{NE}^{\gamma_n}$ ($\text{N}$ and $\text{E}$ denote north and east steps, respectively), then the number of $1$-$\Pop$-sortable elements of $\Tam(\nu)$ is $2^{|\mathscr A(\nu)|}$, where $\mathscr A(\nu)$ is the set of indices $k\in\{0,\ldots,n-1\}$ such that $\gamma_k\geq 1$ (so $|\mathscr A(\nu)|$ is the number of east runs of $\nu$). We then give a recursive procedure for generating the $2$-$\Pop$-sortable elements of $\Tam(\nu)$ for arbitrary $\nu$. This allows us to enumerate $2$-$\Pop$-sortable elements of $\Tam(\nu)$ for a large variety of lattice paths $\nu$. For example, suppose $\nu=\text{E}^{\alpha_0}\text{N}^{\beta_0}\text{E}^{\alpha_1}\text{N}^{\beta_1}\cdots\text{E}^{\alpha_{q-1}}\text{N}^{\beta_{q-1}}\text{E}^{\alpha_{q}}$, where $\alpha_0,\ldots,\alpha_{q},\beta_0,\ldots,\beta_{q-1}$ are nonnegative integers such that $\alpha_1,\ldots,\alpha_{q-1}\geq 1$ and $\beta_0,\ldots,\beta_{q-1}\geq 2$. We prove that the number of $2$-$\Pop$-sortable lattice paths in $\Tam(\nu)$ is $3^{\theta(\nu)}5^{\chi(\nu)}$, where $\theta(\nu)$ is the number of indices $i\in\{0,\ldots,q-1\}$ such that $\alpha_i=1$ and $\chi(\nu)$ is the number of indices $i\in\{0,\ldots,q-1\}$ such that $\alpha_i\geq 2$. This theorem does not apply to $m$-Tamari lattices, so we enumerate their $2$-$\Pop$-sortable elements separately. In particular, we prove that $2$-$\Pop$-sortable elements of Tamari lattices are counted by Pell numbers. 

Our definition of semilattice pop-stack-sorting operators opens the door to several possibilities for future work, especially because there are so many fascinating classes of meet-semilattices one could study. In Section~\ref{SecConclusion}, we collect some specific ideas related to the topics discussed in this article. In particular, we conjecture that for all fixed $m,t\geq 1$, the generating function that counts $t$-$\Pop$-sortable elements of the $m$-Tamari lattices $\Tam_n(m)$ is rational.

\section{$\Pop$-Trivial Meet-Semilattices}\label{SecPopTrivial}

Let us say a meet-semilattice $M$ is \dfn{$\Pop$-trivial} if it has a minimal element $\widehat 0$ and $\Pop_M(x)=\widehat 0$ for every $x\in M$. In this section, we discuss finite $\Pop$-trivial lattices, showing, in particular, that all geometric lattices are $\Pop$-trivial. 

Let $M$ be a finite lattice with minimal element $\widehat 0$. The elements that cover $\widehat 0$ are called \dfn{atoms}. We say $M$ is \dfn{atomic} if every element of $M$ can be written as the join of a set of atoms. We say $M$ is \dfn{graded} if its maximal chains all have the same length. If $M$ is graded, then there is a unique \dfn{rank function} $\text{rk}:M\to\mathbb Z_{\geq 0}$ such that $\text{rk}(\widehat 0)=0$ and $\text{rk}(z)=\text{rk}(y)+1$ for every cover relation $y\lessdot z$. We say $M$ is \dfn{semimodular} if it is graded and $\text{rk}(x)+\text{rk}(y)\geq\text{rk}(x\wedge y)+\text{rk}(x\vee y)$ for all $x,y\in M$. A \dfn{geometric lattice} is a finite graded lattice that is atomic and semimodular. Notable examples of geometric lattices include Boolean lattices, partition lattices, and lattices of subspaces of finite vector spaces \cite[Section~3.4]{Stanley}.  

The importance of geometric lattices comes from matroid theory. Let $E$ be a finite set. A \dfn{matroid} on $E$ is a pair $(E,\mathcal I)$, where $\mathcal I$ is a nonempty collection of subsets of $E$ satisfying the following properties: 
\begin{itemize}
\item If $A\in\mathcal I$ and $B\subseteq A$, then $B\in\mathcal I$.
\item If $A,B\in\mathcal I$ and $|A|<|B|$, then there exists $b\in B\setminus A$ such that $A\cup\{b\}\in\mathcal I$.
\end{itemize}
The sets in $\mathcal I$ are called \dfn{independent}. The \dfn{rank} of a subset $A$ of $E$ (with respect to the matroid), denoted $\text{rk}(A)$, is defined to be the maximum size of an independent set contained in $A$. A \dfn{flat} of the matroid is a subset $A$ of $E$ such that $\text{rk}(A)<\text{rk}(A\cup\{i\})$ for all $i\in E\setminus A$. The \dfn{lattice of flats} of the matroid is the collection of all flats of the matroid ordered by containment; it is indeed a lattice. We say the matroid $(E,\mathcal I)$ is \dfn{simple} if every $2$-element subset of $E$ is independent. It is well known that a lattice is geometric if and only if it is isomorphic to the lattice of flats of a simple finite matroid \cite[Section~3.4]{Stanley}. 

\begin{theorem}
Every geometric lattice is $\Pop$-trivial. 
\end{theorem}

\begin{proof}
Let $M$ be a geometric lattice with minimal element $\widehat 0$. We may assume $M$ has at least $2$ elements since the result is trivial otherwise. We may also assume that $M$ is the lattice of flats of a finite simple matroid $(E,\mathcal I)$. Let $F\in M$ be a flat of this matroid. Choose $a\in F$. Because $(E,\mathcal I)$ is simple, the singleton set $\{a\}$ is independent. We claim that there is an independent subset $A$ of $F$ of cardinality $\text{rk}(F)$ that contains $a$. To see this, choose an arbitrary independent subset $B$ of $F$ of cardinality $\text{rk}(F)$. If $a\in B$, then we can simply put $A=B$. Now assume $a\not\in B$. Note that, by the simplicity of $M$, this implies that $\text{rk}(F)\geq 2$. Since $\{a\}$ is independent, we can use the definition of a matroid to see that there exists $b_1\in B\setminus\{a\}$ such that $\{a,b_1\}$ is independent. If $\text{rk}(F)\geq 3$, then we can use the definition of a matroid again to see that there exists $b_2\in B\setminus\{a,b_1\}$ such that $\{a,b_1,b_2\}$ is independent. Repeating this argument, we eventually find distinct elements $b_1,\ldots,b_{\text{rk}(F)-1}\in B$ such that $\{a,b_1,\ldots,b_{\text{rk}(F)-1}\}$ is independent, proving the claim. 

Let $F'\in M$ be the smallest flat containing $A\setminus\{a\}$. Then $\text{rk}(F')=\text{rk}(A\setminus\{a\})=\text{rk}(F)-1$, and $F'$ does not contain $a$. The flat $F$ covers the flat $F'$ in $M$. It follows that $\Pop_M(F)\subseteq F'$. Hence, $a$ is not an element of $\Pop_M(F)$. As $a$ was arbitrary, this shows that $\Pop_M(F)=\emptyset=\widehat 0$. 
\end{proof}

Let us remark that there are finite graded atomic lattices that are not $\Pop$-trivial. There are also finite semimodular lattices that are not $\Pop$-trivial. For example, the lattice whose Hasse diagram is depicted in Figure~\ref{FigMeet1} is semimodular but not $\Pop$-trivial. The dual of this lattice is atomic but not $\Pop$-trivial. 

\begin{figure}[ht]
  \begin{center}{\includegraphics[height=2.727cm]{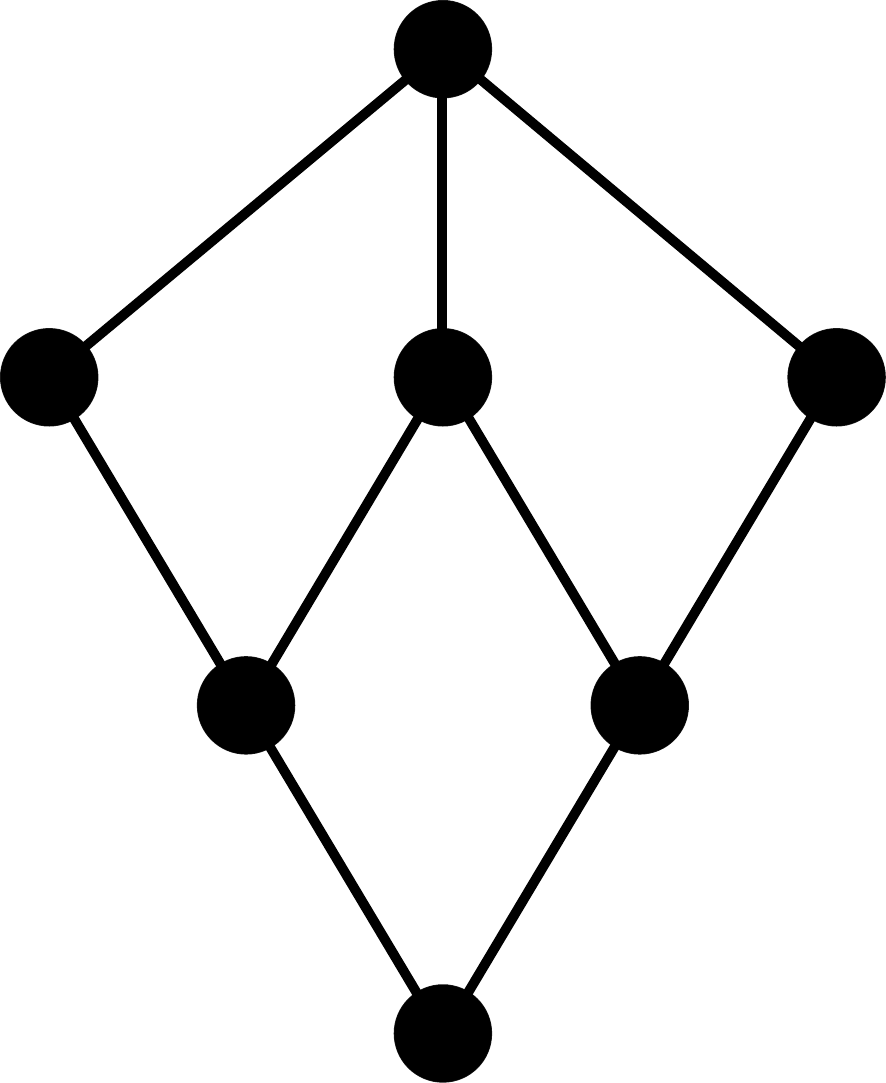}}
  \end{center}
  \caption{A semimodular lattice that is not $\Pop$-trivial. The dual of this lattice is atomic but not $\Pop$-trivial.}\label{FigMeet1}
\end{figure}

\begin{problem}\label{ProbMeet1}
Characterize finite $\Pop$-trivial lattices. 
\end{problem}

\section{Sublattices and Lattice Congruences}\label{SecCongruences}

When working with the semilattice pop-stack-sorting operator $\Pop_M$, the first challenge is to obtain a combinatorial description of what the operator is actually doing. In this section, we provide a general theorem that, when specialized to the case of Tamari lattices, provides such a description.  

A poset is called \dfn{locally finite} if each of its closed intervals is finite. A \dfn{lattice congruence} on a lattice $M$ is an equivalence relation $\equiv$ on $M$ such that if $x_1\equiv x_2$ and $y_1\equiv y_2$, then $x_1\wedge y_1\equiv x_2\wedge y_2$ and $x_1\vee y_1\equiv x_2\vee y_2$. Lattice congruences have received an enormous amount of attention; a standard reference for the subject is \cite{Gratzer}. If $M$ is a locally finite lattice with a minimal element and $\equiv$ is a lattice congruence on $M$, then each congruence class of $\equiv$ is a closed interval. For each $x\in M$, we denote by $\pi_\downarrow(x)$ the minimal element of the congruence class containing $x$. The map $\pi_\downarrow$ is order-preserving and idempotent. A \dfn{sublattice} of $M$ is a subset $M'\subseteq M$ such that $x\wedge y\in M'$ and $x\vee y\in M'$ for all $x,y\in M'$. 

\begin{lemma}\label{LemMeet1}
If $\equiv$ is a lattice congruence on a locally finite lattice $M$ such that the set $M'=\{\pi_\downarrow(x):x\in M\}$ is a sublattice of $M$, then $\pi_\downarrow\left(\bigwedge A\right)=\bigwedge\{\pi_\downarrow(a):a\in A\}$ for every finite set $A\subseteq M$.  
\end{lemma}

\begin{proof}
The map $\pi_\downarrow:M\to M$ is order-preserving, so $\pi_\downarrow(\bigwedge A)\leq\pi_\downarrow(a)$ for all $a\in A$. Hence, $\pi_\downarrow(\bigwedge A)\leq\bigwedge\{\pi_\downarrow(a):a\in A\}$. To prove the reverse inequality, first note that $\bigwedge\{\pi_\downarrow(a):a\in A\}\leq \bigwedge A$. This implies that $\pi_\downarrow(\bigwedge\{\pi_\downarrow(a):a\in A\})\leq \pi_\downarrow(\bigwedge A)$. Since $M'$ is a sublattice of $M$, we have $\bigwedge\{\pi_\downarrow(a):a\in A\}\in M'$. The map $\pi_\downarrow$ is idempotent, so it acts as the identity on $M'$. Hence, $\pi_\downarrow(\bigwedge\{\pi_\downarrow(a):a\in A\})=\bigwedge\{\pi_\downarrow(a):a\in A\}$. This shows that $\bigwedge\{\pi_\downarrow(a):a\in A\}\leq\pi_\downarrow(\bigwedge A)$.
\end{proof}

\begin{theorem}\label{ThmMeet1}
Let $M$ be a locally finite lattice with a minimal element. Let $\equiv$ be a lattice congruence on $M$ such that the set $M'=\{\pi_\downarrow(x):x\in M\}$ is a sublattice of $M$. Then $\Pop_{M'}(x)=\pi_\downarrow(\Pop_M(x))$ for all $x\in M'$. 
\end{theorem}

\begin{proof}
Fix $x\in M'$. If $x$ is the minimal element of $M$ (equivalently, of $M'$), then certainly $\Pop_{M'}(x)=x=\pi_\downarrow(\Pop_M(x))$. Hence, we may assume $x$ is not the minimal element of $M$. The hypothesis that $M$ is locally finite guarantees that $\{y\in M:y\lessdot x\}$ is finite. According to Lemma~\ref{LemMeet1}, we have \[\pi_\downarrow(\Pop_M(x))=\pi_\downarrow\left(\bigwedge\{y\in M:y\lessdot x\}\right)=\bigwedge\{\pi_\downarrow(y):y\lessdot x\}.\] Since $M'$ is a sublattice of $M$, the set $\{\pi_\downarrow(y):y\lessdot x\}$ has the same meet in $M$ as in $M'$. On the other hand, $\Pop_{M'}(x)=\bigwedge\{z\in M':z\lessdot_{M'} x\}$, where $z\lessdot_{M'}x$ means $z$ is covered by $x$ in $M'$. Thus, it suffices to prove that $\{\pi_\downarrow(y):y\lessdot x\}=\{z\in M':z\lessdot_{M'} x\}$. 

Suppose $y\in M$ is such that $y\lessdot x$. Then $\pi_\downarrow(y)\in M'$ and $\pi_\downarrow(y)<x$. We want to prove that $\pi_\downarrow(y)\lessdot_{M'}x$. Assume, by way of contradiction, that there exists $x'\in M'$ such that $\pi_\downarrow(y)<x'<x$. We have $x'\not\equiv\pi_\downarrow(y)$, so $x'$ cannot lie in the closed interval between $\pi_\downarrow(y)$ and $y$. Hence, $x'\not\leq y$. This means that $x'\vee y>y$. Since $x>x'$ and $x>y$, we have $x\geq x'\vee y>y$. As $x$ covers $y$, this implies that $x'\vee y=x$. Since $\pi_\downarrow(y)\equiv y$ and $\equiv$ is a lattice congruence, $x'=x'\vee \pi_\downarrow(y)\equiv x'\vee y=x$. However, this means that $x$ is not the minimal element of its congruence class, contradicting the fact that it is in $M'$. This proves that $\pi_\downarrow(y)\lessdot_{M'}x$. Consequently, $\{\pi_\downarrow(y):y\lessdot x\}\subseteq\{z\in M':z\lessdot_{M'} x\}$. 

To prove the reverse containment, suppose $z\in M'$ is covered by $x$ in $M'$. We must show that $z=\pi_\downarrow(y)$ for some $y\in M$ such that $y\lessdot x$ in $M$. Since $z<x$ in $M$, there exists $y\in M$ such that $z\leq y\lessdot x$. We have $z=\pi_\downarrow(z)\leq \pi_\downarrow(y)\leq y<x$, so the relations $z\leq\pi_\downarrow(y)<x$ hold in $M'$. As $z$ is covered by $x$ in $M'$, this implies that $z=\pi_\downarrow(y)$.   
\end{proof}

\begin{remark}
Theorem~\ref{ThmMeet1} concerns the downward projection map $\pi_\downarrow$ and the semilattice pop-stack-sorting operator $\Pop_M$. Each of these would still be defined if we only assumed $\equiv$ was a semilattice congruence (i.e., an equivalence relation compatible with meets, but not necessarily joins). However, the hypothesis that $\equiv$ is a lattice congruence is necessary in Theorem~\ref{ThmMeet1}. Indeed, consider the following $4$-element lattice $M$: 
\[\begin{array}{l}\includegraphics[height=2.1cm]{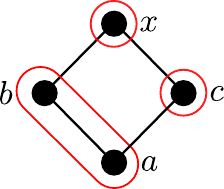}.\end{array}\] Let $\equiv$ be the semilattice congruence on $M$ whose congruence classes are circled. The chain $M'=\{a<c<x\}$ is a sublattice of $M$. However, $\equiv$ is not a lattice congruence, and $\Pop_{M'}(x)=c\neq a=\pi_\downarrow(\Pop_M(x))$.  
\end{remark}

When the hypotheses of Theorem~\ref{ThmMeet1} are satisfied, we can transfer information about $\Pop_M$ into information about $\Pop_{M'}$. For example, if $y\in M$ is in a singleton congruence class of $\equiv$, then $\Pop_{M'}^{-1}(y)=M'\cap\Pop_M^{-1}(y)$. For a second example, suppose $M$ is the right weak order on a finite irreducible Coxeter group $W$. Following \cite{DefantCoxeterPop}, we say a map $f:M\to M$ is \dfn{compulsive} if $f(x)\leq\Pop_M(x)$ for all $x\in M$. One of the main theorems from \cite{DefantCoxeterPop} states that if $f:M\to M$ is compulsive, then every forward orbit of $M$ has size at most the Coxeter number of $W$. Suppose $\equiv$ is a lattice congruence on $M$ such that the set $M'=\{\pi_\downarrow(x):x\in M\}$ is a sublattice of $M$. The map $\pi_\downarrow\circ\Pop_M:M\to M$ is compulsive, and its restriction to $M'$ is $\Pop_{M'}$ by Theorem~\ref{ThmMeet1}. Therefore, every forward orbit of $\Pop_{M'}$ has size at most the Coxeter number of $W$. 

\medskip 

Throughout the rest of this section, we explain how to use Theorem~\ref{ThmMeet1} to provide a combinatorial description of the semilattice pop-stack-sorting operators on Tamari lattices. Here, we view the Tamari lattices as the restrictions of the right weak orders on symmetric groups to $312$-avoiding permutations; in Section~\ref{SecTamari}, we give a different definition of Tamari lattices in terms of Dyck paths. 

Let $S_n$ denote the set of permutations of the set $[n]:=\{1,\ldots,n\}$; we write a permutation $w\in S_n$ as a word $w(1)\cdots w(n)$ in one-line notation. A \dfn{left inversion} of $w$ is a pair $(i,j)$ such that $1\leq i<j\leq n$ and $w^{-1}(i)>w^{-1}(j)$. The \dfn{right weak order} on $S_n$ is the partial order $\leq_R$ given by saying $v\leq_R w$ if every left inversion of $v$ is also a left inversion of $w$. It is well known that the right weak order is a lattice \cite{BjornerBrenti}. A \dfn{descending run} of a permutation $w$ is a maximal decreasing subsequence of $w$. For example, the descending runs of $4258617$ are $42$, $5$, $861$, and $7$. The \dfn{pop-stack-sorting map} is the operator on $S_n$ that reverses the descending runs of a permutation while keeping entries in different descending runs in the same relative order. For instance, it sends $4258617$ to $2451687$. As discussed in \cite{DefantCoxeterPop}, the pop-stack-sorting map on $S_n$ is equal to the semilattice pop-stack-sorting operator $\Pop_{S_n}$ associated to the right weak order on $S_n$.

We say entries $c,a,b$ form a \dfn{$312$-pattern} in a permutation $w$ if they appear in the order $c,a,b$ in $w$ and $a<b<c$. For example, the entries $6,1,3$ form a $312$-pattern in $264513$. We say $w$ is \dfn{$312$-avoiding} if no three entries form a $312$ pattern in $w$. Let $\Av_n(312)$ denote the set of $312$-avoiding permutations in $S_n$. Bj\"orner and Wachs \cite{BjornerWachsShellable} proved that $\Av_n(312)$ is a sublattice of the right weak order on $S_n$ isomorphic to the Tamari lattice $\Tam_n$. Define a relation $\rhd$ on $S_n$ by saying $w\rhd v$ if we can write $w=XcaYbZ$ and $v=XacYbZ$, where $X,Y,Z$ are words and $a,b,c$ are numbers such that $c,a,b$ form a $312$-pattern in $w$. Observe that if $w\rhd v$, then $v\leq_R w$. Let $\equiv_{\syl}$ be the reflexive, symmetric, and transitive closure of $\rhd$. The equivalence relation $\equiv_{\syl}$ is called the \dfn{sylvester congruence} on (the right weak order of) $S_n$ \cite{Hivert}. Reading \cite{ReadingCambrian} observed that $\equiv_{\syl}$ is a lattice congruence on the right weak order of $S_n$. Furthermore, every sylvester class contains a unique $312$-avoiding permutation, which is the minimal element of the congruence class. Thus, if we are given $w\in S_n$, then we can compute $\pi_\downarrow(w)$ (the unique $312$-avoiding permutation that is sylvester-equivalent to $w$) by constructing a chain $w\rhd v_1\rhd v_2\rhd\cdots\rhd v_k=\pi_\downarrow(w)$. According to Theorem~\ref{ThmMeet1}, we have 
\begin{equation}\label{EqMeet1}
\Pop_{\Av_n(312)}(w)=\pi_\downarrow\left(\Pop_{S_n}(w)\right)
\end{equation}
for all $w\in \Av_n(312)$, where $\Pop_{\Av_n(312)}$ is the semilattice pop-stack-sorting operator on the Tamari lattice $\Av_n(312)$. 

As an application of \eqref{EqMeet1}, let us determine which elements of $\Av_n(312)$ have the largest forward orbits under $\Pop_{\Av_n(312)}$.

\begin{proposition}\label{PropMeet1}
For $n\geq 2$ and $w\in\Av_n(312)$, we have \[\left|O_{\Pop_{\Av_n(312)}}(w)\right|\leq n\] where equality holds if and only if $w$ ends in the suffix $n1$. The number of elements of $\Av_n(312)$ that end in the suffix $n1$ is the Catalan number $C_{n-2}=\frac{1}{n-1}\binom{2(n-2)}{n-2}$. 
\end{proposition}

\begin{proof}
Define the \dfn{tail length} of a permutation $v\in S_n$, denoted $\tl(v)$, to be the largest integer $k\in\{0,\ldots,n\}$ such that $v(i)=i$ for all $i\in\{n-k+1,\ldots,n\}$. It is straightforward to check that if we have a relation $u\rhd v$, then $\tl(u)\leq\tl(v)$ and $u^{-1}(1)\geq v^{-1}(1)$. Hence, 
\begin{equation}\label{EqMeet2}
\tl(u)\leq\tl(\pi_\downarrow(u))\quad\text{and}\quad u^{-1}(1)\geq (\pi_\downarrow(u))^{-1}(1)
\end{equation} for all $u\in S_n$. 

Consider $u_0\in \Av_n(312)$. To ease notation, let $u_t=\Pop_{\Av_n(312)}^t(u_0)$ and $v_t=\Pop_{S_n}(u_t)$. By \eqref{EqMeet1}, we have $u_{t+1}=\pi_\downarrow(v_t)$ for all $t\geq 0$. Suppose $u_0$ has tail length $k<n-1$. We can write $u_0=X(n-k)Y(n-k+1)(n-k+2)\cdots n$ for some words $X$ and $Y$. Because $u_0$ is $312$-avoiding, $Y$ must be decreasing. This shows that $(n-k)Y$ is a descending run of $u_0$, so $v_0=\Pop_{S_n}(u_0)$ ends in the suffix $\rev(Y)(n-k)(n-k+1)\cdots n$, where $\rev(Y)$ is the reverse of $Y$. Hence, $\tl(v_0)\geq k+1>\tl(u_0)$. By combining this observation with \eqref{EqMeet2}, we find that $\tl(u_1)>\tl(u_0)$. Applying the same argument to $u_1$ shows that if $\tl(u_1)<n-1$, then $\tl(u_2)>\tl(u_1)$. By repeating this argument inductively, we find that $\tl(u_t)\geq t$ for all $0\leq t\leq n-1$. In particular, $\tl(u_{n-1})\geq n-1$. The only permutation in $S_n$ whose tail length is at least $n-1$ is the identity permutation $e=123\cdots n$, so $u_{n-1}=\Pop_{\Av_n(312)}^{n-1}(u_0)=e$. Since $e$ is a fixed point of $\Pop_{\Av_n(312)}$, the forward orbit of $u_0$ has size at most $n$ (this is also a consequence of the result about compulsive maps mentioned above since the Coxeter number of $S_n$ is $n$). 

Now suppose the forward orbit of $u_0$ under $\Pop_{\Av_n(312)}$ has size $n$. Then $u_{n-2}\neq u_{n-1}=e$. We saw above that $\tl(u_{n-2})\geq n-2$, so we must have $u_{n-2}=21345\cdots n$. In particular, $u_{n-2}^{-1}(1)=2$. If $u_t$ does not start with the entry $1$, then the entry $1$ will move to the left when we apply $\Pop_{S_n}$ to $u_t$. Appealing to \eqref{EqMeet2}, we find that \[u_{t+1}^{-1}(1)=\left(\pi_\downarrow(v_t)\right)^{-1}(1)\leq v_t^{-1}(1)<u_t^{-1}(1)\] whenever $u_t^{-1}(1)>1$. Since $u_{n-2}^{-1}(1)=2$, this means we must have $u_t^{-1}(1)=n-t$ and $v_t^{-1}(1)=n-t-1$ for all $0\leq t\leq n-2$. In particular, $u_0^{-1}(1)=n$ and $v_0^{-1}(1)=n-1$. This implies that the last descending run of $u_0$ has size $2$. Because $u_0$ is $312$-avoiding, its last descending run starts with $n$, so $u_0$ ends with the suffix $n1$. 

We now prove the converse. For $0\leq k\leq n-2$, let $Z_k$ be the set of permutations in $\Av_n(312)$ of the form $X(n-k)1(n-k+1)(n-k+2)\cdots n$ for some word $X$. Note that $Z_0$ is the set of permutations in $\Av_n(312)$ that end in the suffix $n1$; our goal is to prove that each element of $Z_0$ has a forward orbit under $\Pop_{\Av_n(312)}$ of size $n$. Equivalently, we need to prove that $e\not\in\Pop_{\Av_n(312)}^{n-2}(Z_0)$. Suppose $w=X(n-k)1(n-k+1)(n-k+2)\cdots n\in Z_k$ for some $k\leq n-3$. Because $w$ is $312$-avoiding, the last descending run in $X$ starts with the entry $n-k-1$. It follows that $\Pop_{S_n}(w)$ is of the form $Y(n-k-1)1(n-k)(n-k+1)\cdots n$. If we apply the map $\pi_\downarrow$ to $\Pop_{S_n}(w)$ by constructing a chain $\Pop_{S_n}(w)\rhd w_1\rhd w_2\rhd\cdots\rhd \pi_\downarrow(\Pop_{S_n}(w))$, then there will not be any step in the chain when the entries $(n-k-1)1(n-k)(n-k+1)\cdots n$ move. In other words, the permutation $\Pop_{\Av_n(312)}=\pi_\downarrow(\Pop_{S_n}(w))$ is of the form $Y'(n-k-1)1(n-k)(n-k+1)\cdots n$ for some word $Y'$. This shows that $\Pop_{\Av_n(312)}$ maps $Z_k$ into $Z_{k+1}$. As this is true for all $0\leq k\leq n-3$, we find that $\Pop_{\Av_n(312)}^{n-2}(Z_0)\subseteq Z_{n-2}=\{21345\cdots n\}$, so $e\not\in\Pop_{\Av_n(312)}^{n-2}(Z_0)$. 

We have proven the first statement of the proposition. We are left to enumerate the permutations in $Z_0$. Given a permutation $w=Xn1\in Z_0$, let $\iota(w)$ be the permutation in $\Av_{n-2}(312)$ obtained by decreasing each of the entries in $X$ by $1$. The map $\iota$ is a bijection from $Z_0$ to $\Av_{n-2}(312)$. It is well known that $|\!\Av_{n-2}(312)|$ is the Catalan number $C_{n-2}$.    
\end{proof}

\begin{remark}
As discussed in \cite{DefantCoxeterPop, DefantCoxeterStack}, one can define West's stack-sorting map $\mathtt{s}$ (see \cite{BonaSurvey, West} for its usual definition) by $\mathtt{s}(w)=w\circ\pi_\downarrow(w^{-1})$, where $\circ$ is the group operation (i.e., composition) in $S_n$ and $\pi_\downarrow$ denotes the downward projection map associated to the sylvester congruence on the right weak order (the articles \cite{DefantCoxeterPop,DefantCoxeterStack} use the left weak order instead, so we have transferred from the left to the right weak order by taking inverses of permutations). Combining this fact with \eqref{EqMeet1}, we find that the semilattice pop-stack-sorting operator $\Pop_{\Av_n(312)}$ on the Tamari lattice $\Av_n(312)$ can be described in terms of the pop-stack-sorting map and West's stack-sorting map. More precisely, \[\Pop_{\Av_n(312)}(w)=\Pop_{S_n}(w)\circ\mathtt{s}\left(\Pop_{S_n}(w)^{-1}\right). \qedhere\]
\end{remark}

\begin{remark}
Other notable instances where Theorem~\ref{ThmMeet1} applies are provided by Cambrian congruences on the weak orders of finite Coxeter groups. Indeed, the minimal elements of the congruence classes of a Cambrian congruence form a sublattice of the weak order \cite{ReadingSortable}. 
\end{remark}

\section{$\nu$-Tamari Lattices}\label{SecTamari}

\subsection{Lattice Paths}\label{SubsecLattice}
In this article, a \dfn{lattice path} is a finite path in the plane that starts at the origin and uses $(0,1)$ steps and $(1,0)$ steps. A $(0,1)$ step is called a \dfn{north step} and is denoted by $\text{N}$, whereas a $(1,0)$ step is called an \dfn{east step} and is denoted by $\text{E}$. We identify lattice paths with finite words over the alphabet $\{\text{E},\text{N}\}$. We often use exponents to denote concatenation of words; for example, $(\text{NE}^3)^2=\text{NEEENEEE}$. Given a lattice path $\nu$, we write $\Tam(\nu)$ for the set of lattice paths lying weakly above $\nu$ that have the same endpoints as $\nu$. For example, if $\nu=\text{ENNEEEENNE}$ is the lattice path shown in black in Figure~\ref{FigMeet2}, then the path $\mu=\text{NENENEEENE}$ shown in red is in $\Tam(\nu)$. 

\begin{figure}[ht]
  \begin{center}{\includegraphics[height=4.2cm]{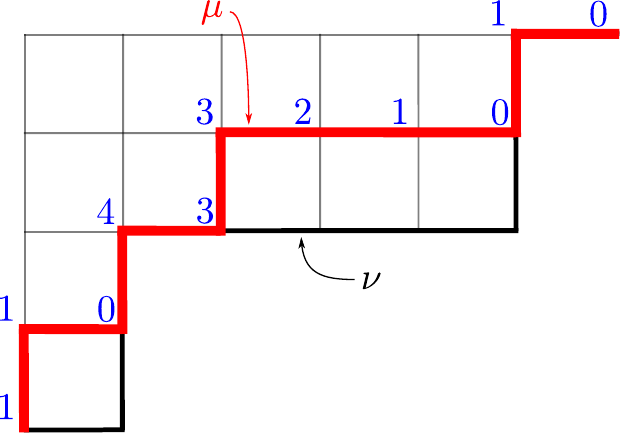}}
  \end{center}
  \caption{The lattice path $\mu=\text{NENENEEENE}$ in the $\nu$-Tamari lattice $\Tam(\nu)$, where $\nu=\text{ENNEEEENNE}$. Each lattice point $p$ is labeled with its horizontal distance. }\label{FigMeet2}
\end{figure}

Fix a lattice path $\nu$, and let $\mu\in\Tam(\nu)$. For each lattice point $p$ on $\mu$, define the \dfn{horizontal distance} of $p$ to be the maximum number of east steps that can be taken, starting at $p$, before crossing $\nu$. See Figure~\ref{FigMeet2} for an illustration of this definition. Now suppose $p$ is a lattice point on $\mu$ that is preceded by an east step and followed by a north step in $\mu$. Let $p'$ be the first lattice point on $\mu$ that appears after $p$ and has the same horizontal distance as $p$. If we let $D_{[p,p']}$ be the subpath of $\mu$ that starts at $p$ and ends at $p'$, then we can write $\mu=X\text{E}D_{[p,p']}Y$ for some lattice paths $X$ and $Y$. Let $\mu'=XD_{[p,p']}EY$. Then $\mu'\in \Tam(\nu)$. In this case, we write $\mu\lessdot \mu'$. See Figure~\ref{FigMeet3}. 

\begin{figure}[ht]
  \begin{center}{\includegraphics[height=3.923cm]{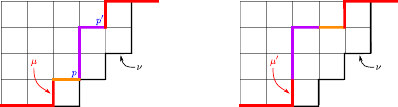}}
  \end{center}
  \caption{The lattice path $\mu$ on the left is covered by the lattice path $\mu'$ on the right in $\Tam(\nu)$.}\label{FigMeet3}
\end{figure}

Let $\leq$ be the partial order on $\Tam(\nu)$ whose cover relations are the relations of the form $\mu\lessdot \mu'$ described in the previous paragraph. Pr\'eville-Ratelle and Viennot introduced this partial order in \cite{PrevilleViennot}, where they proved that the poset $(\Tam(\nu),\leq)$ is a lattice. In the special case when $\nu=(\text{NE}^m)^n$, the lattice $\Tam(\nu)$ is the $n^\text{th}$ \dfn{$m$-Tamari lattice}, which we denote by $\Tam_n(m)$. The elements of $\Tam_n(m)$ are called \dfn{$m$-ballot paths}. In the very special case when $\nu=(\text{NE})^n$, $\Tam(\nu)$ is the classical $n^\text{th}$ \dfn{Tamari lattice}, which we denote by $\Tam_n$. The elements of $\Tam_n$ are \dfn{Dyck paths}.  

Throughout this section, we will focus exclusively on the semilattice pop-stack-sorting operators on $\nu$-Tamari lattices, so we will often omit the subscript and write $\Pop$ instead of $\Pop_{\Tam(\nu)}$. 

\subsection{$\nu$-Bracket Vectors}
One crucial tool in our analysis is a reformulation of $\nu$-Tamari lattices due to Ceballos, Padrol and Sarmiento \cite{Ceballos2}. In what follows, we make the convention that vectors are $0$-indexed. For example, we would say that the number $5$ appears in positions $0$ and $3$ in the vector $(5,2,6,5,6)$.  

\begin{definition}\label{DefMeet2}
Fix a lattice path $\nu$ that starts at $(0,0)$ and ends at $(\ell-n,n)$. Let ${\bf b}(\nu)=(b_0(\nu),\ldots,b_{\ell}(\nu))$ be the vector obtained by reading the heights (i.e., $y$-coordinates) of the lattice points on $\nu$ in the order they appear in $\nu$. For $0\leq k\leq n$, let $f_k$ be the maximum index such that $b_{f_k}(\nu)=k$. We call $f_0,\ldots,f_n$ the \dfn{fixed positions} of $\nu$. A \dfn{$\nu$-bracket vector} is an integer vector $\vec{\mathsf{b}}=(\mathsf b_0,\ldots,\mathsf b_{\ell})$ such that:
\begin{enumerate}[(I)]
\item $\mathsf b_{f_k}=k$ for all $0\leq k\leq n$; 
\item $b_i(\nu)\leq \mathsf b_i\leq n$ for all $0\leq i\leq \ell$;
\item\label{Item3} if $\mathsf b_i=k$, then $\mathsf b_j\leq k$ for all $i+1\leq j\leq f_k$. 
\end{enumerate} 
Let $\Vecc(\nu)$ denote the set of $\nu$-bracket vectors.  
\end{definition}

It is often convenient to note that condition \eqref{Item3} can be replaced with the condition that $\vec{\mathsf b}$ avoids the pattern $121$, which means that there do not exist indices $i_1<i_2<i_3$ such that $\mathsf b_{i_1}=\mathsf b_{i_3}<\mathsf b_{i_2}$. It is also useful to keep in mind that if $\vec{\mathsf b}=(\mathsf b_0,\ldots,\mathsf b_\ell)\in\Vecc(\nu)$, then for each $0\leq k\leq n$, the sequence $\mathsf b_{f_{k-1}+1},\ldots,\mathsf b_{f_k}$ is weakly decreasing. 

Observe that if $\vec{\mathsf{b}}=(\mathsf b_0,\ldots,\mathsf b_{\ell})$ and $\vec{\mathsf{b}}'=(\mathsf b'_0,\ldots,\mathsf b'_{\ell})$ are $\nu$-bracket vectors, then so is their componentwise minimum $\min(\vec{\mathsf{b}},\vec{\mathsf{b}}')=(\min\{\mathsf b_0,\mathsf b_0'\},\ldots,\min\{\mathsf b_{\ell},\mathsf b_{\ell}'\})$. We make the convention throughout the rest of the paper that $f_{-1}=-1$. 

Suppose we are given a path $\mu\in\Tam(\nu)$. Let us write down a tuple $(\underline{\hphantom{1}},\underline{\hphantom{1}},\ldots,\underline{\hphantom{1}})$ with $\ell+1$ empty slots. We now traverse the path $\mu$. Whenever we arrive at a lattice point with height $k$, we insert a $k$ into the rightmost unoccupied slot that is not strictly to the right of the $f_k^\text{th}$ position. The end result of this procedure is a vector ${\bf b}(\mu)=(b_0(\mu),\ldots,b_{\ell}(\mu))$. For example, if $\mu=\nu$, then the vector ${\bf b}(\mu)$ constructed via this process is the same as the vector ${\bf b}(\nu)$ given in Definition~\ref{DefMeet2}. Although the vector ${\bf b}(\mu)$ depends on the fixed lattice path $\nu$, we suppress this dependence in the notation. 

\begin{example}
Suppose $\nu=\text{ENNEEEENNE}$ and $\mu=\text{NENENEEENE}$ are as depicted in Figure~\ref{FigMeet2}. Then ${\bf b}(\nu)=(0,0,1,2,2,2,2,2,3,4,4)$. The heights of the lattice points on $\mu$ are $0,1,1,2,2,3,3,3,3,$ $4,4$. The vectors appearing in the construction of ${\bf b}(\mu)$ are \[\hphantom{\to }\,(\underline{\hphantom{1}},\underline{\hphantom{0}},\underline{\hphantom{1}},\underline{\hphantom{3}},\underline{\hphantom{3}},\underline{\hphantom{3}},\underline{\hphantom{2}},\underline{\hphantom{2}},\underline{\hphantom{3}},\underline{\hphantom{4}},\underline{\hphantom{4}})\to (\underline{\hphantom{1}},0,\underline{\hphantom{1}},\underline{\hphantom{3}},\underline{\hphantom{3}},\underline{\hphantom{3}},\underline{\hphantom{2}},\underline{\hphantom{2}},\underline{\hphantom{3}},\underline{\hphantom{4}},\underline{\hphantom{4}})\to (\underline{\hphantom{1}},0,1,\underline{\hphantom{3}},\underline{\hphantom{3}},\underline{\hphantom{3}},\underline{\hphantom{2}},\underline{\hphantom{2}},\underline{\hphantom{3}},\underline{\hphantom{4}},\underline{\hphantom{4}})\hphantom{.}\] \[\to (1,0,1,\underline{\hphantom{3}},\underline{\hphantom{3}},\underline{\hphantom{3}},\underline{\hphantom{2}},\underline{\hphantom{2}},\underline{\hphantom{3}},\underline{\hphantom{4}},\underline{\hphantom{4}})\to (1,0,1,\underline{\hphantom{3}},\underline{\hphantom{3}},\underline{\hphantom{3}},\underline{\hphantom{2}},2,\underline{\hphantom{3}},\underline{\hphantom{4}},\underline{\hphantom{4}})\to (1,0,1,\underline{\hphantom{3}},\underline{\hphantom{3}},\underline{\hphantom{3}},2,2,\underline{\hphantom{3}},\underline{\hphantom{4}},\underline{\hphantom{4}})\hphantom{.}\]
\[\to (1,0,1,\underline{\hphantom{3}},\underline{\hphantom{3}},\underline{\hphantom{3}},2,2,3,\underline{\hphantom{4}},\underline{\hphantom{4}})\to (1,0,1,\underline{\hphantom{3}},\underline{\hphantom{3}},3,2,2,3,\underline{\hphantom{4}},\underline{\hphantom{4}})\to (1,0,1,\underline{\hphantom{3}},3,3,2,2,3,\underline{\hphantom{4}},\underline{\hphantom{4}})\hphantom{.}\]
\[\to (1,0,1,3,3,3,2,2,3,\underline{\hphantom{4}},\underline{\hphantom{4}})\to (1,0,1,3,3,3,2,2,3,\underline{\hphantom{4}},4)\to (1,0,1,3,3,3,2,2,3,4,4).\]
Thus, ${\bf b}(\mu)=(1,0,1,3,3,3,2,2,3,4,4)$. 
\end{example}

The utility of these definitions arises from the next theorem, which is due to Ceballos, Padrol, and Sarmiento. 

\begin{theorem}[\cite{Ceballos2}]\label{ThmBracket}
For each $\mu\in\Tam(\nu)$, the vector ${\bf b}(\mu)$ is a $\nu$-bracket vector. The map ${\bf b}:\Tam(\nu)\to\Vecc(\nu)$ is a bijection. Furthermore, for any paths $\mu,\mu'\in\Tam(\nu)$, we have ${\bf b}(\mu\wedge\mu')=\min({\bf b}(\mu),{\bf b}(\mu'))$. 
\end{theorem}

\begin{figure}[ht]
  \begin{center}{\includegraphics[height=14cm]{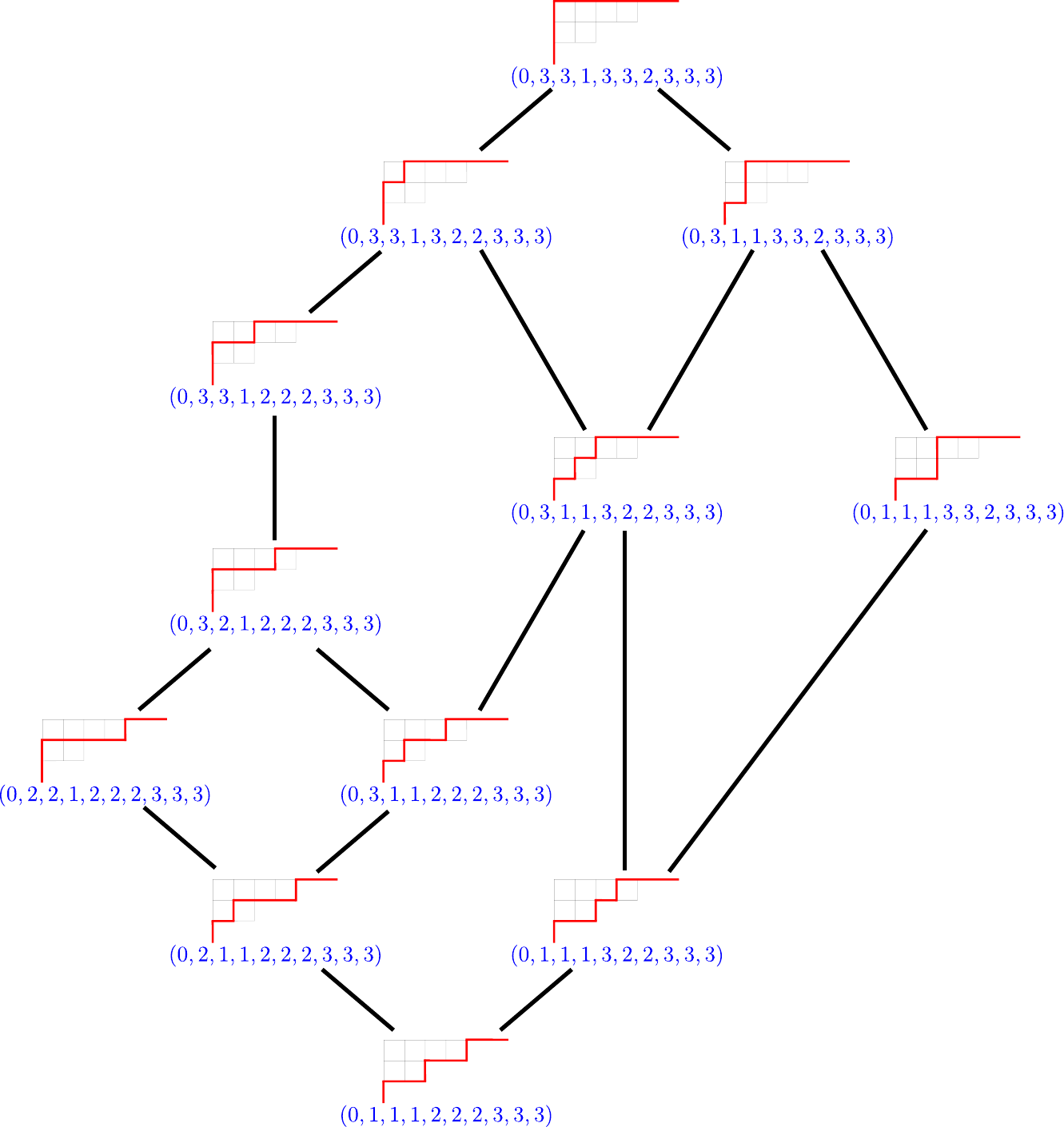}}
  \end{center}
  \caption{The $2$-Tamari lattice $\Tam_3(2)=\Tam(\nu)$, where $\nu=(\text{NE}^2)^3$. Each $2$-ballot path appears above its associated $\nu$-bracket vector.}\label{FigMeet4}
\end{figure}

The previous theorem allows us to transfer the lattice structure of $\Tam(\nu)$ to a lattice structure on $\Vecc(\nu)$, where the meet operation is given by the componentwise minimum. The partial order on $\Vecc(\nu)$ is given by saying $(\mathsf b_0,\ldots,\mathsf b_\ell)\leq(\mathsf b_0',\ldots,\mathsf b_\ell')$ if $\mathsf b_i\leq \mathsf b_i'$ for all $0\leq i\leq \ell$.  

Using the previous theorem, we can give a description of what the map $\Pop$ does in terms of $\nu$-bracket vectors. Let $\mu\in\Tam(\nu)$. Define $\Cov(\mu)$ to be the set of $\nu$-bracket vectors $\vec{\mathsf{b}}$ that are covered by ${\bf b}(\mu)$ in $\Vecc(\nu)$. Let $\Delta(\mu)$ be the set of indices $i\in\{0,\ldots,\ell\}$ such that $b_i(\mu)>b_i(\nu)$ and $b_j(\mu)<b_i(\mu)$ for all $i+1\leq j\leq f_{b_i(\nu)}$. Suppose $i\in\Delta(\mu)$, and let $k=b_i(\nu)$. This means that $f_{k-1}<i\leq f_k$. Because $b_i(\mu)>b_i(\nu)=k$, we must actually have $f_{k-1}<i< f_k$. If we let $k'=b_{i+1}(\mu)$, then $k=b_{i+1}(\nu)\leq k'\leq b_i(\mu)-1$, and it follows from the condition \eqref{Item3} in Definition~\ref{DefMeet2} that $b_j(\mu)\leq k'$ for all $i+1\leq j\leq f_{k'}$. This shows that it makes sense to define $\eta_i(\mu)$ to be the largest element of $\{b_i(\nu),\ldots,b_i(\mu)-1\}$ such that $b_j(\mu)\leq \eta_i(\mu)$ for all $i+1\leq j\leq f_{\eta_i(\mu)}$. In addition, this argument shows that $\eta_i(\mu)\geq b_{i+1}(\mu)$. On the other hand, if $i'\in\{0,\ldots,\ell\}\setminus\Delta(\mu)$, then we simply define $\eta_{i'}(\mu)=b_{i'}(\mu)$. 

\begin{proposition}\label{LemMeet2}
Preserve the notation from above. For $i\in\Delta(\mu)$, let ${\bf b}_\downarrow^i(\mu)$ be the vector obtained from ${\bf b}(\mu)$ by replacing the entry $b_i(\mu)$ in position $i$ with $\eta_i(\mu)$. We have \[\Cov(\mu)=\left\{{\bf b}_\downarrow^i(\mu):i\in\Delta(\mu)\right\}\quad\text{and}\quad{\bf b}\left(\Pop(\mu)\right)=\left(\eta_0(\mu),\ldots,\eta_\ell(\mu)\right).\]
\end{proposition} 

\begin{proof}
Observe that $\left(\eta_0(\mu),\ldots,\eta_\ell(\mu)\right)$ is the componentwise minimum of the vectors ${\bf b}_\downarrow^i(\mu)$ for $i\in\Delta(\mu)$. Therefore, if we can prove that $\Cov(\mu)=\left\{{\bf b}_\downarrow^i(\mu):i\in\Delta(\mu)\right\}$, then it will follow from Theorem~\ref{ThmBracket} that ${\bf b}\left(\Pop(\mu)\right)=\left(\eta_0(\mu),\ldots,\eta_\ell(\mu)\right)$. 

First, choose $i\in\Delta(\mu)$. Because ${\bf b}(\mu)$ is a $\nu$-bracket vector, it follows from the definition of $\eta_i(\mu)$ that ${\bf b}_\downarrow^i(\mu)$ is also a $\nu$-bracket vector. Also, ${\bf b}_\downarrow^i(\mu)<{\bf b}(\mu)$ in the lattice $\Vecc(\nu)$. We want to show that ${\bf b}_\downarrow^i(\mu)\lessdot{\bf b}(\mu)$. Suppose otherwise. Then there exists a $\nu$-bracket vector $\vec{\mathsf{b}}^*=(\mathsf b_0^*,\ldots,\mathsf b_\ell^*)$ with ${\bf b}_\downarrow^i(\mu)<\vec{\mathsf b}^*<{\bf b}(\mu)$. We have $\mathsf b_\gamma^*=b_\gamma(\mu)$ for all $\gamma\in\{0,\ldots,\ell\}\setminus\{i\}$, and $\eta_i(\mu)<\mathsf b_i^*<b_i(\mu)$. By the definition of $\eta_i(\mu)$, there exists an index $j$ such that $b_j(\mu)>\mathsf b_i^*$ and $i+1\leq j\leq f_{\mathsf b_i^*}$. However, this contradicts the fact that $\vec{\mathsf b}^*$ is a $\nu$-bracket vector. Consequently, ${\bf b}_\downarrow^i(\mu)\in\Cov(\mu)$. 

To finish the proof, we need to show that $\Cov(\mu)\subseteq\left\{{\bf b}_\downarrow^i(\mu):i\in\Delta(\mu)\right\}$. Consider a vector $\vec{\mathsf{b}}=(\mathsf b_0,\ldots,\mathsf b_\ell)\in\Cov(\mu)$. We have $\mathsf b_j\leq b_j(\mu)$ for all $j$. Let $i$ be the smallest index such that $\mathsf b_i<b_i(\mu)$. Let $\vec{\mathsf{b}}'$ be the vector obtained from $\vec{\mathsf{b}}$ by replacing the entry $\mathsf b_i$ in position $i$ with $b_i(\mu)$. Because $\vec{\mathsf{b}}$ and ${\bf b}(\mu)$ satisfy the three conditions from Definition~\ref{DefMeet2}, it is straightforward to check that $\vec{\mathsf{b}}'$ also satisfies those conditions (in particular, changing $\vec{\mathsf b}$ into $\vec{\mathsf b}'$ cannot create any $121$ patterns). This means that $\vec{\mathsf{b}}'\in\Vecc(\nu)$. Furthermore, $\vec{\mathsf{b}}<\vec{\mathsf{b}}'\leq{\bf b}(\mu)$, so we must have $\vec{\mathsf{b}}'={\bf b}(\mu)$ since ${\bf b}(\mu)$ covers $\vec{\mathsf{b}}$. Therefore, $\vec{\mathsf b}$ agrees with ${\bf b}(\mu)$ in every position except the $i^\text{th}$ position. Let $k=b_i(\nu)$. We know by Definition~\ref{DefMeet2} that $b_i(\mu)>\mathsf b_i\geq k$ and that $b_j(\mu)=\mathsf b_j\leq \mathsf b_i<b_i(\mu)$ for all $i+1\leq j\leq f_{\mathsf b_i}$. In particular, $b_i(\mu)>k$, and $b_j(\mu)<b_i(\mu)$ for all $i+1\leq j\leq f_k$. This shows that $i\in\Delta(\mu)$. We saw in the previous paragraph that ${\bf b}_\downarrow^i(\mu)$ is a $\nu$-bracket vector covered by ${\bf b}(\mu)$. The vectors $\vec{\mathsf b}$, ${\bf b}(\mu)$, and ${\bf b}_\downarrow^i(\mu)$ agree in all positions except the $i^\text{th}$ position, so they must all be comparable in $\Vecc(\nu)$. Since ${\bf b}(\mu)$ covers both $\vec{\mathsf{b}}$ and ${\bf b}_\downarrow^i(\mu)$, we must have $\vec{\mathsf{b}}={\bf b}_\downarrow^i(\mu)$.   
\end{proof}

\begin{example}
Let $\nu=\text{NEENENEENE}$ and $\mu=\text{NNEENENEEE}$. Then \[{\bf b}(\nu)=(0,1,1,1,2,2,3,3,3,4,4)\quad\text{and}\quad{\bf b}(\mu)=(0,4,2,1,2,2,4,3,3,4,4).\] Note that $b_1(\mu)=4>1=b_1(\nu)$. Furthermore, $f_{b_1(\nu)}=f_1=3$, and $b_j(\mu)<4$ for all $2\leq j\leq 3$. This shows that $1\in\Delta(\mu)$. Now, $\eta_1(\mu)$ is defined to be the largest element of $\{1,2,3\}$ such that $b_j(\mu)\leq \eta_1(\mu)$ for all $2\leq j\leq f_{\eta_1(\mu)}$. Note that $\eta_1(\mu)\neq 3$ since $b_6(\mu)=4>3$ and $2\leq 6\leq 8=f_3$. On the other hand, $f_2=5$, and $b_j(\mu)\leq 2$ for all $2\leq j\leq 5$. This shows that $\eta_1(\mu)=2$, so ${\bf b}_\downarrow^1(\mu)=(0,2,2,1,2,2,4,3,3,4,4)$. Similar analysis shows that $\Delta(\mu)=\{1,2,6\}$, $\eta_2(\mu)=1$, and $\eta_6(\mu)=3$. Also, $\eta_i(\mu)=b_i(\mu)$ for $i\in\{0,\ldots,10\}\setminus\{1,2,6\}$ by definition. Finally, \[{\bf b}\left(\Pop(\mu)\right)=\left(\eta_0(\mu),\ldots,\eta_\ell(\mu)\right)=(0,2,1,1,2,2,3,3,3,4,4),\] so $\Pop(\mu)=\text{NENEENEENE}$. 
\end{example}

\begin{corollary}\label{CorMeet1}
Suppose $\mu\in\Tam(\nu)$, $0\leq k\leq n$, and $f_{k-1}<i<f_k$. Then $b_{i}(\Pop(\mu))\geq b_{i+1}(\mu)$. 
\end{corollary}
\begin{proof}
When defining $\eta_i(\mu)$, we observed that $\eta_i(\mu)\geq b_{i+1}(\mu)$. Now apply Proposition~\ref{LemMeet2}. 
\end{proof}

By iterating Corollary~\ref{CorMeet1}, we obtain the following additional corollary. 
\begin{corollary}\label{LemMeet5}
Let $\mu\in\Tam(\nu)$ and $0\leq k\leq n$ be such that $f_k\geq f_{k-1}+2$ and such that $b_j(\mu)=n$ for all $f_{k-1}<j<f_k$. Then $b_{f_{k-1}+1}(\Pop^{f_k-f_{k-1}-2}(\mu))=n$. 
\end{corollary}

\subsection{Forward Orbits of Maximum Size}

We are going to use Proposition~\ref{LemMeet2} to determine the maximum size of a forward orbit of $\Pop_{\Tam(\nu)}$. First, let us establish some lemmas and introduce a little more terminology. 

\begin{lemma}\label{LemMeet3}
Let $\mu\in\Tam(\nu)$, where $\nu$ is as above. We have $i\in\Delta(\mu)$ if and only if $i<\ell$ and $b_i(\mu)>b_{i+1}(\mu)$. If $i\in\Delta(\mu)$ and $i-1$ is not a fixed position of $\nu$, then $i-1\in\Delta(\Pop(\mu))$. 
\end{lemma}

\begin{proof}
Suppose $f_{k-1}<i\leq f_k$. It follows from Definition~\ref{DefMeet2} that the numbers $b_j(\mu)$ for $f_{k-1}< j\leq f_k$ are weakly decreasing and that $b_{f_k}(\mu)=k=b_{f_k}(\nu)$. Furthermore, if $k<n$, then $k<b_{f_k+1}(\mu)$. This shows that $i\in\Delta(\mu)$ if and only if $b_i(\mu)>b_{i+1}(\mu)$. Assume $i-1$ is not a fixed position (i.e., $i-1\neq f_{k-1}$). Then $\eta_{i-1}(\mu)\geq b_i(\mu)$ (by the same argument used to prove Corollary~\ref{CorMeet1}). By Proposition~\ref{LemMeet2}, we have $b_{i-1}(\Pop(\mu))=\eta_{i-1}(\mu)$ and $b_i(\mu)>b_i(\Pop(\mu))$. This shows that $b_{i-1}(\Pop(\mu))>b_i(\Pop(\mu))$, so $i-1\in\Delta(\Pop(\mu))$.  
\end{proof}

Let $\vec{\mathsf{b}}=(\mathsf b_0,\ldots, \mathsf b_\ell)$ be a $\nu$-bracket vector. Let $\beta$ be an index such that $f_{k-1}<\beta<f_k$, where $0\leq k\leq n$ (with $f_{-1}=-1$). We say $\beta$ is \dfn{unimpeded} with respect to $\vec{\mathsf{b}}$ if there exists an index $\gamma$ with $\beta\leq\gamma\leq f_k$ such that $\mathsf b_\beta>\mathsf b_{\beta+1}>\cdots>\mathsf b_\gamma=k$. In particular, if $\mathsf b_{\beta+1}=k$, then $\beta$ is unimpeded with respect to $\vec{\mathsf b}$ because we can either take $\gamma=\beta$ or $\gamma=\beta+1$. Our terminology is motivated by the following lemma, which tells us that the entry in position $\beta$ drops when we apply $\Pop$ so long as $\beta$ is unimpeded; in other words, the entry is not impeded from going down. 

\begin{lemma}\label{LemUnimpeded}
Let $\mu\in\Tam(\nu)$. Let $\beta$ and $k$ be such that $0\leq k\leq n$ and $f_{k-1}<\beta<f_k$, and suppose $\beta$ is unimpeded with respect to ${\bf b}(\mu)$. Then $\beta$ is unimpeded with respect to ${\bf b}(\Pop(\mu))$. If $b_\beta(\mu)>k$, then $b_\beta(\Pop(\mu))<b_\beta(\mu)$. If $\beta>f_{k-1}+1$, then $\beta-1$ is also unimpeded with respect to ${\bf b}(\Pop(\mu))$. 
\end{lemma}

\begin{proof}
Let $\gamma$ be the index such that $\beta\leq\gamma\leq f_k$ and $b_\beta(\mu)>b_{\beta+1}(\mu)>\cdots>b_\gamma(\mu)=k$. We proceed by induction on $\gamma-\beta$. First, suppose $\gamma-\beta=0$. It follows from the definition of a $\nu$-bracket vector and Proposition~\ref{LemMeet2} that $k\leq b_\beta(\Pop(\mu))\leq b_\beta(\mu)=k$, so $b_\beta(\Pop(\mu))=k$. This implies that $\beta$ is unimpeded with respect to ${\bf b}(\Pop(\mu))$. If $\beta>f_{k-1}+1$, then, since $b_{\beta}(\Pop(\mu))=k$, the index $\beta-1$ is also unimpeded with respect to ${\bf b}(\Pop(\mu))$. 

Now assume $\gamma-\beta\geq 1$. Then $b_\beta(\mu)>b_{\beta+1}(\mu)\geq k$, so Lemma~\ref{LemMeet3} tells us that $\beta\in\Delta(\mu)$. According to Proposition~\ref{LemMeet2}, we have $b_\beta(\Pop(\mu))=\eta_\beta(\mu)<b_\beta(\mu)$. Furthermore, the argument used to prove Corollary~\ref{CorMeet1} yields the inequality $\eta_\beta(\mu)\geq b_{\beta+1}(\mu)$. By induction, we either have $b_{\beta+1}(\mu)=k$ or $b_{\beta+1}(\Pop(\mu))<b_{\beta+1}(\mu)$. In the first case, $b_{\beta+1}(\Pop(\mu))=k$, so $\beta$ is unimpeded with respect to ${\bf b}(\Pop(\mu))$. In the second case, $b_{\beta+1}(\Pop(\mu))<b_{\beta}(\Pop(\mu))$, so we can use the induction hypothesis that $\beta+1$ is unimpeded with respect to ${\bf b}(\Pop(\mu))$ to see that $\beta$ is unimpeded with respect to ${\bf b}(\Pop(\mu))$. Finally, suppose $\beta>f_{k-1}+1$. Then Lemma~\ref{LemMeet3} tells us that $\beta-1\in\Delta(\Pop(\mu))$ and $b_{\beta-1}(\Pop(\mu))>b_\beta(\Pop(\mu))$. Since $\beta$ is unimpeded with respect to ${\bf b}(\Pop(\mu))$, so is $\beta-1$. 
\end{proof}

\begin{theorem}\label{ThmMeet2}
Let $\nu$ be a lattice path that starts at $(0,0)$, ends at $(\ell-n,n)$, and has fixed positions $f_0,\ldots,f_n$. Let $K(\nu)$ be the set of indices $k\in\{0,\ldots,n-1\}$ such that $f_k-f_{k-1}\geq 2$. If $K(\nu)=\emptyset$, let $\Theta(\nu)=1$; otherwise, let $\Theta(\nu)=\max\limits_{k\in K(\nu)}(f_k-f_{k-1}-k+n-1)$. Then \[\max_{\mu\in\Tam(\nu)}\left|O_{\Pop_{\Tam(\nu)}}(\mu)\right|=\Theta(\nu).\] 
\end{theorem}

\begin{proof}
If $K(\nu)=\emptyset$, then $\nu=\text{N}^n\text{E}^{\ell-n}$, so $\Tam(\nu)=\{\nu\}$. In this case, the only forward orbit of $\Pop$ has size $\Theta(\nu)=1$. Hence, in what follows, we may assume $K(\nu)\neq\emptyset$. 

Choose $\mu\in\Tam(\nu)$ and $0\leq k\leq n$. If $k\not\in K(\nu)$ and $f_{k-1}<j\leq f_k$, then either $j=f_k$ or $k=n$, so every $\nu$-bracket vector has the entry $b_j(\nu)$ in its $j^\text{th}$ position. On the other hand, we claim that if $k\in K(\nu)$, then $b_{f_{k-1}+1}(\Pop^{f_k-f_{k-1}-k+n-2}(\mu))=k$. If we can prove this claim, then it will follow from Definition~\ref{DefMeet2} that $b_j(\Pop^{f_k-f_{k-1}-k+n-2}(\mu))=k=b_j(\nu)$ for every $k\in K(\nu)$ and every $f_{k-1}<j\leq f_k$. It will then follow that $\Pop^{\Theta(\nu)-1}(\mu)=\nu$. As $\mu$ was arbitrary and $\Pop(\nu)=\nu$, this will prove that every element of $\Tam(\nu)$ has a forward orbit under $\Pop$ of size at most $\Theta(\nu)$. 

Now fix $k\in K(\nu)$. Since $b_{f_k}(\mu)=k$, the index $f_k-1$ is unimpeded with respect to ${\bf b}(\mu)$. If $f_k-2>f_{k-1}$, then it follows from Lemma~\ref{LemUnimpeded} that $f_k-2$ is unimpeded with respect to ${\bf b}(\Pop(\mu))$. If $f_k-3>f_{k-1}$, then we can use Lemma~\ref{LemUnimpeded} again to see that $f_k-3$ is unimpeded with respect to ${\bf b}(\Pop^2(\mu))$. Repeating this argument, we eventually find that $f_{k-1}+1$ is unimpeded with respect to ${\bf b}(\Pop^{f_k-f_{k-1}-2}(\mu))$. Lemma~\ref{LemUnimpeded} now tells us that $f_{k-1}+1$ is unimpeded with respect to ${\bf b}(\Pop^t(\mu))$ for every $t\geq f_k-f_{k-1}-2$. The same lemma tells us that if $t\geq f_k-f_{k-1}-2$, then $b_{f_{k-1}+1}(\Pop^{t+1}(\mu))$ is either equal to $k$ or is strictly less than $b_{f_{k-1}+1}(\Pop^t(\mu))$. Since $b_{f_{k-1}+1}(\Pop^{f_k-f_{k-1}-2}(\mu))\leq n$, we have $b_{f_{k-1}+1}(\Pop^{f_k-f_{k-1}-2+n-k}(\mu))=k$, as desired. 

We now want to show that the upper bound we have just proven is tight. As before, we may assume $K(\nu)\neq\emptyset$. Fix $k\in\{0,\ldots,n\}$ such that $f_k-f_{k-1}-k+n-1=\Theta(\nu)$. Let $\vec{\mathsf b}^*=(\mathsf b_0^*,\ldots,\mathsf b_\ell^*)$ be such that $\mathsf b_j^*=n$ for all $j\in\{f_{k-1}+1,\ldots,f_k-1\}$ and $\mathsf b_j^*=b_j(\nu)$ for all $j\in\{0,\ldots,\ell\}\setminus\{f_{k-1}+1,\ldots,f_k-1\}$. Then $\vec{\mathsf{b}}^*$ is a $\nu$-bracket vector, so we know by Theorem~\ref{ThmBracket} that there exists $\mu^*\in\Tam(\nu)$ such that ${\bf b}(\mu^*)=\vec{\mathsf{b}}^*$. By repeatedly applying Proposition~\ref{LemMeet2}, one can straightforwardly describe ${\bf b}(\Pop^t(\mu^*))$ explicitly for every $t\geq 0$ (see Example~\ref{ExamMeet1}). First, we have $b_j(\Pop^t(\mu^*))=b_j(\nu)$ for every $j\in\{0,\ldots,\ell\}\setminus\{f_{k-1}+1,\ldots,f_k-1\}$. Second, if $j\in\{f_{k-1}+1,\ldots,f_k-1\}$, then 
\[b_j(\Pop^t(\mu^*))=\begin{cases} \max\{k,n-(t+1-(f_k-j))\}, & \mbox{if } t\geq f_k-j-1; \\ n, & \mbox{if } t\leq f_k-j-2. \end{cases}\] In particular, suppose we set $t=\Theta(\nu)-2$ and $j=f_{k-1}+1$. Since $k\in K(\nu)$, we have $k\leq n-1$, so $t=f_k-f_{k-1}-k+n-3\geq f_k-f_{k-1}-(n-1)+n-3=f_k-j-1$. Hence,
\[b_{f_{k-1}+1}(\Pop^{\Theta(\nu)-2}(\mu^*))=b_{f_{k-1}+1}(\Pop^{f_k-f_{k-1}-k+n-3}(\mu))\] \[=\max\{k,n-(f_k-f_{k-1}-k+n-2-(f_k-f_{k-1}-1))\}=\max\{k,k+1\}=k+1>k=b_j(\nu).\] This shows that $\Pop^{\Theta(\nu)-2}(\mu^*)\neq\nu$, so $\left|O_{\Pop}(\mu^*)\right|\geq\Theta(\nu)$.  
\end{proof}

\begin{example}\label{ExamMeet1}
Let us illustrate the formula for ${\bf b}(\Pop^t(\mu^*))$ given in the second half of the proof of Theorem~\ref{ThmMeet2}. Suppose $\nu=\text{NEEEENEN}$. Then $n=3$, $\ell=8$, and ${\bf b}(\nu)=(0,1,1,1,1,1,2,2,3)$. Then $K(\nu)=\{1,2\}$, and the maximum value of $f_k-f_{k-1}-k+n-1$ for $k\in K(\nu)$ is $\Theta(\nu)=6$, which is attained when $k=1$. Now, $\mu^*$ is the lattice path such that ${\bf b}(\mu)=\vec{\mathsf{b}}^*=(0,3,3,3,3,1,2,2,3)$ (specifically, $\mu=\text{NNENEEEE}$). The sequence of $\nu$-bracket vectors ${\bf b}(\Pop^t(\mu^*))$ for $t\geq 0$ is \[(0,3,3,3,3,1,2,2,3),\quad (0,3,3,3,2,1,2,2,3),\quad (0,3,3,2,1,1,2,2,3),\quad (0,3,2,1,1,1,2,2,3),\] \[ (0,2,1,1,1,1,2,2,3),\quad (0,1,1,1,1,1,2,2,3),\quad (0,1,1,1,1,1,2,2,3),\quad\ldots.\qedhere\] 
\end{example}

Recall that the $m$-Tamari lattice $\Tam_n(m)$ is the $\nu$-Tamari lattice associated to the lattice path $\nu=(\text{NE}^m)^n$. When we work with the $m$-Tamari lattice $\Tam_n(m)$, we preserve the definitions from above with the tacit understanding that $\nu=(\text{NE}^m)^n$. In this case, $f_k=k(m+1)$ for all $0\leq k\leq n$. For notational convenience, we will write $J_k=\{(k-1)(m+1)+1,\ldots,k(m+1)-1\}$ (suppressing the dependence on $m$). In Theorem~\ref{ThmMeet3}, we will give an explicit description of the elements of $\Tam_n(m)$ whose forward orbits under $\Pop$ are of the maximum size. In Theorem~\ref{ThmMeet4}, we will give a simple formula for the number of such lattice paths. First, we require a technical lemma. 

\begin{lemma}\label{LemMeet4}
Let $m,n\geq 1$, and let $\mu\in\Tam_n(m)$. Fix $1\leq k\leq n$. We have $b_j(\Pop^t(\mu))\leq\max\{k,n+m-1-t\}$ whenever  $j\in J_k$ and $t\geq m-1$. If $j\in J_k\setminus\{(k-1)(m+1)+1\}$ and $t\geq m-1$, then $b_j(\Pop^t(\mu))\leq\max\{k,n+m-2-t\}$. If there exists an index $r\in J_k$ such that $b_r(\mu)<n$, then $b_j(\Pop^t(\mu))\leq\max\{k,n+m-2-t\}$ whenever $j\in J_k$ and $t\geq m-1$. 
\end{lemma}

\begin{proof}
We have $f_{k}=k(m+1)$ for all $0\leq k\leq n$. Let $\mathcal N=b_{(k-1)(m+1)+1}(\mu)$. Then $b_j(\mu)\leq \mathcal N$ for all $j\in J_k$. The index $f_k-1$ is unimpeded with respect to ${\bf b}(\mu)$, so we can use Lemma~\ref{LemUnimpeded} to see that $f_k-2$ and $f_k-1$ are both unimpeded with respect to ${\bf b}(\Pop(\mu))$. Repeating this argument shows that all of the indices in $J_k$ are unimpeded with respect to ${\bf b}(\Pop^{m-1}(\mu))$. Applying Lemma~\ref{LemUnimpeded} repeatedly again, we find that all of the indices in $J_k$ are unimpeded with respect to ${\bf b}(\Pop^t(\mu))$ for all $t\geq m-1$. In addition, we find that $b_j(\Pop^t(\mu))\leq\max\{k,\mathcal N+m-1-t\}$ for all $j\in J_k$ and $t\geq m-1$. This proves the first desired statement because $\mathcal N\leq n$. The second desired statement follows from a similar line of reasoning because all of the elements of $J_k\setminus\{(k-1)(m+1)+1\}$ are unimpeded with respect to ${\bf b}(\Pop^{m-2}(\mu))$. 

Now suppose $r\in J_k$ is such that $b_r(\mu)<n$. If $\mathcal N\leq n-1$, then the same argument we used in the previous paragraph proves the third desired statement. Therefore, we may assume $\mathcal N=n$. This implies that there exists an index $i$ such that $(k-1)(m+1)+1\leq i\leq k(m+1)-2$ and $b_i(\mu)>b_{i+1}(\mu)$. By Lemma~\ref{LemMeet3}, $i\in\Delta(\mu)$. Applying Lemma~\ref{LemMeet3} repeatedly, we find that there is an integer $q$ with $0\leq q\leq m-2$ such that $(k-1)(m+1)+1\in\Delta(\Pop^q(\mu))$. Then $b_{(k-1)(m+1)+1}(\Pop^{m-1}(\mu))\leq b_{(k-1)(m+1)+1}(\Pop^{q+1}(\mu))<b_{(k-1)(m+1)+1}(\Pop^{q}(\mu))\leq n$, where the strict inequality comes from Proposition~\ref{LemMeet2}. It follows that $b_j(\Pop^{m-1}(\mu))\leq n-1$ for all $j\in J_k$. We saw above that all of the indices in $J_k$ are unimpeded with respect to ${\bf b}(\Pop^t(\mu))$ for all $t\geq m-1$. Repeated applications of Lemma~\ref{LemUnimpeded} yield the desired fact that $b_j(\Pop^t(\mu))\leq\max\{k,n+m-2-t\}$ for all $j\in J_k$ and $t\geq m-1$.   
\end{proof}

\begin{theorem}\label{ThmMeet3}
Let $m\geq 1$ and $n\geq 2$. We have \[\max_{\mu\in\Tam_n(m)}\left|O_{\Pop_{\Tam_n(m)}}(\mu)\right|=m+n-1.\] A lattice path $\mu\in\Tam_n(m)$ satisfies $\left|O_{\Pop_{\Tam_n(m)}}(\mu)\right|=m+n-1$ if and only if $b_m(\mu)=n$ and $b_{k(m+1)-1}(\mu)<n$ for all $2\leq k\leq n-1$.
\end{theorem}

\begin{proof}
Let $\nu=(\text{NE}^m)^n$. We have $f_k=k(m+1)$ for all $0\leq k\leq n$, so $K(\nu)=\{1,\ldots,n-1\}$. Hence, $\Theta(\nu)=\max\limits_{1\leq k\leq n-1}(m+1-k+n-1)=m+n-1$. The first assertion in the theorem now follows from Theorem~\ref{ThmMeet2}. 

Let $\mu\in\Tam_n(m)$ be such that $b_m(\mu)=n$ and $b_{k(m+1)-1}(\mu)<n$ for all $2\leq k\leq n-1$. According to Lemma~\ref{LemMeet4}, $b_j(\Pop^t(\mu))\leq\max\{k,n+m-2-t\}$ whenever $2\leq k\leq n-1$, $j\in J_k$, and $t\geq m-1$. The same lemma tells us that $b_j(\Pop^t(\mu))\leq\max\{1,n+m-2-t\}$ whenever $2\leq j\leq m$ and $t\geq m-1$. Also, $b_j(\Pop^t(\mu))=b_j(\nu)$ for all $(n-1)(m+1)\leq j\leq n(m+1)$ and $t\geq 0$ since ${\bf b}(\Pop^t(\mu))$ is a $\nu$-bracket vector. In summary, $b_j(\Pop^t(\mu))\leq\max\{b_j(\nu),n+m-2-t\}$ for all $2\leq j\leq (m+1)n$ and $t\geq m-1$. If we set $k=1$ in Corollary~\ref{LemMeet5}, then we find that $b_1(\Pop^{m-1}(\mu))=n$. Because $b_j(\Pop^{m-1}(\mu))\leq\max\{b_j(\nu),n-1\}$ for all $2\leq j\leq (m+1)n$, we have $\eta_1(\Pop^{m-1}(\mu))=\max\{1,n-1\}$. By Proposition~\ref{LemMeet2}, 
$b_1(\Pop^m(\mu))=\max\{1,n-1\}$. Because $b_j(\Pop^m(\mu))\leq\max\{b_j(\nu),n-2\}$ for all $2\leq j\leq (m+1)n$, we have $\eta_1(\Pop^m(\mu))=\max\{1,n-2\}$. By Proposition~\ref{LemMeet2}, 
$b_1(\Pop^m(\mu))=\max\{1,n-2\}$. Repeating this argument, we find that $b_1(\Pop^t(\mu))=\max\{1,n+m-1-t\}$ for all $t\geq m-1$. In particular, $b_1(\Pop^{m+n-3}(\mu))=2>1=b_1(\nu)$, so $\Pop^{m+n-3}(\mu)\neq\nu$. We deduce that $\left|O_{\Pop}(\mu)\right|\geq m+n-1$. Therefore, $\left|O_{\Pop}(\mu)\right|=m+n-1$ by the first statement of the theorem.  

Now assume $\mu\in\Tam_n(m)$ is such that $\left|O_{\Pop}(\mu)\right|=m+n-1$; we will show that ${\bf b}(\mu)$ is of the specified form. Invoking Lemma~\ref{LemMeet4}, we find that $b_j(\Pop^{m+n-3}(\mu))=b_j(\nu)$ for all $2\leq j\leq(m+1)n$. Furthermore, $b_0(\Pop^{m+n-3}(\mu))=0=b_0(\nu)$. We have $\Pop^{m+n-3}(\mu)\neq\nu$, so $b_1(\Pop^{m+n-3}(\mu))>b_1(\nu)=1$. Lemma~\ref{LemMeet4} tells us that if there is an index $r\in J_1$ such that $b_r(\mu)<n$, then $b_1(\Pop^{m+n-3}(\mu))\leq 1$, which is impossible. Hence, $b_m(\mu)=n$. It follows from Lemma~\ref{LemUnimpeded} that the indices in $J_1$ are all unimpeded with respect to ${\bf b}(\Pop^t(\mu))$ for all $t\geq m-1$. The same lemma tells us that for $t\geq m-1$, the values of $b_1(\Pop^t(\mu))$ strictly decrease as $t$ increases until they stabilize at $1$. Since $b_1(\Pop^{m+n-3}(\mu))\geq 2$ and $b_1(\Pop^{m-1}(\mu))\leq n$, the only way this can happen is if $b_1(\Pop^t(\mu))=\max\{1,n+m-1-t\}$ for all $t\geq m-1$. In particular, $b_1(\Pop^{m-1}(\mu))=n$ and $b_1(\Pop^m(\mu))=n-1$. Proposition~\ref{LemMeet2} tells us that $1\in\Delta(\Pop^{m-1}(\mu))$ and $\eta_1(\Pop^{m-1}(\mu))=n-1$. This means that $b_j(\Pop^{m-1}(\mu))\leq n-1$ for all $2\leq j\leq f_{n-1}$. Now choose an integer $k$ such that $2\leq k\leq n-1$. We have $2\leq f_{k-1}+1\leq f_{n-1}$, so $b_{f_{k-1}+1}(\Pop^{m-1}(\mu))\leq n-1$. According to Corollary~\ref{LemMeet5}, there must be an index $j\in J_k$ such that $b_j(\mu)<n$. Consequently, $b_{k(m+1)-1}(\mu)<n$. 
\end{proof}

Our next goal is to prove that the number of lattice paths $\mu\in\Tam_n(m)$ such that $\left|O_{\Pop}(\mu)\right|=m+n-1$ is $\frac{1}{n-1}\binom{(m+1)(n-2)+m-1}{n-2}$. First, we need to consider a different family of lattice paths counted by the same numbers. Let us say an $m$-ballot path $\mu\in\Tam_n(m)$ is \dfn{primitive} if it does not touch the line $y=x/m$ except at its endpoints $(0,0)$ and $(mn,n)$. Let $\Tam_n'(m)$ be the set of primitive elements of $\Tam_n(m)$. Every $m$-ballot path can be decomposed uniquely as a concatenation of primitive $m$-ballot paths, so we have $A_m(z)=\dfrac{B_m(z)}{1-B_m(z)}$, where \[A_m(z)=\sum_{n\geq 1}|\!\Tam_n(m)|z^n\quad\text{and}\quad B_m(z)=\sum_{n\geq 1}|\!\Tam_n'(m)|z^n.\] It is well known that $|\!\Tam_n(m)|=\frac{1}{mn+1}\binom{(m+1)n}{n}$ and that $A_m(z)=z(1+A_m(z))^{m+1}$. Once we replace $A_m(z)$ with $\dfrac{B_m(z)}{1-B_m(z)}$ and simplify, this equation becomes $B_m(z)=\dfrac{z}{(1-B_m(z))^m}$. A straightforward application of the Lagrange inversion formula \cite[Theorem~5.4.2]{Stanley2} then yields 
\begin{equation}\label{EqMeet3}
|\!\Tam_n'(m)|=\frac{1}{n}\binom{(m+1)(n-1)+m-1}{n-1}.
\end{equation}

In order to connect primitive $m$-ballot paths with $m$-ballot paths whose forward orbits under $\Pop$ are of maximum size, we will employ \dfn{generating trees}, which were introduced in \cite{Chung} and have been used frequently afterward \cite{Banderier, West3, West2}. To describe a generating tree for a class of combinatorial objects, we first label the objects and specify a scheme by which each object of size $n$ can be uniquely generated from an object of size $n-1$. The generating tree consists of an \dfn{axiom} that specifies the labels of the objects of size $1$ along with a \dfn{rule} that describes the labels of the objects generated by each object with a given label. For example, in the generating tree 
\[\text{Axiom: }(2)\qquad\text{Rule: }(1)\leadsto(2),\quad(2)\leadsto(1)(2),\] the axiom $(2)$ tells us that we begin with a single object of size $1$ that has label $2$. The rule $(1)\leadsto(2),\hspace{.15cm}(2)\leadsto(1)(2)$ tells us that each object of size $n-1$ with label $1$ generates a single object of size $n$ with label $2$, whereas each object of size $n-1$ with label $2$ generates one object of size $n$ with label $1$ and one object of size $n$ with label $2$. This classical example of a generating tree describes objects counted by the Fibonacci numbers \cite[Example~3]{West3}. 

Fix a positive integer $m$. We want to provide a generating tree for the class of primitive $m$-ballot paths. Instead, we will provide a generating tree for a class of words that are in bijection with primitive $m$-ballot paths. Suppose $\mu\in\Tam_n(m)$. Observe that $b_0(\mu)=0$. It is straightforward to check that $\mu$ is primitive if and only if $b_1(\mu)=n$. Let ${\bf u}(\mu)$ be the word $u_2(\mu)\cdots u_{(m+1)n}(\mu)$ over the alphabet $[n]$, where $u_i(\mu)=n+1-b_{i}(\mu)$. For example, if $m=2$, $n=3$, and ${\bf b}(\mu)=(0,3,2,1,2,2,2,3,3,3)$, then ${\bf u}(\mu)=23222111$. Let $\mathscr U_n(m)=\{{\bf u}(\mu):\mu\in\Tam_n'(m)\}$. The map ${\bf u}$ is a bijection from $\Tam_n'(m)$ to $\mathscr U_n(m)$. By Definition~\ref{DefMeet2}, a word $\mathsf u=\mathsf u_2\cdots \mathsf u_{(m+1)n}$ over the alphabet $[n]$ is in $\mathscr U_n(m)$ if and only if it satisfies the following conditions: 
\begin{itemize}
\item $\mathsf u_{k(m+1)}=n+1-k$ for all $1\leq k\leq n$;
\item $1\leq \mathsf u_i\leq n+1-k$ whenever $1\leq k\leq n$ and $(k-1)(m+1)+1\leq i\leq k(m+1)$;
\item if $\mathsf u_i=n+1-k$, then $\mathsf u_j\geq n+1-k$ for all $i+1\leq j\leq k(m+1)$.
\end{itemize}
Observe that the third condition could equivalently be replaced by the condition that $\mathsf u$ avoids the pattern $212$; this means that there do not exist indices $i_1<i_2<i_3$ such that $\mathsf u_{i_1}=\mathsf u_{i_3}>\mathsf u_{i_2}$. 

For $\mathsf u=\mathsf u_2\cdots \mathsf u_{(m+1)n}\in\mathscr U_n(m)$, let us say a letter $j$ is \dfn{blocked} in $\mathsf u$ if there exists a letter $i<j$ that appears to the left of an occurrence of $j$ in $\mathsf u$. Define the \dfn{label} of $\mathsf u$ to be the number of elements of $[n+1]$ that are not blocked in $\mathsf u$. 

The advantage of the class of words $\bigcup_{n\geq 1}\mathscr U_n(m)$ comes from the observation that if $n\geq 2$ and $\mathsf{u}\in\mathscr U_n(m)$, then the word $\mathsf{u'}$ obtained by deleting the prefix of $\mathsf{u}$ of length $m+1$ is in $\mathscr U_{n-1}(m)$. We say the word $\mathsf{u'}$ \dfn{generates} the word $\mathsf{u}$. For example, the word $14133222111\in\mathscr U_4(2)$ is generated by the word $33222111\in\mathscr U_3(2)$. Since $\mathsf u$ avoids the pattern $212$, there cannot exist a letter that appears in the prefix of $\mathsf u$ of length $m+1$ and that is also blocked in $\mathsf u'$. 

Consider $\mathsf v'\in\mathscr U_{n-1}(m)$, and let $a_1<\cdots<a_r$ be the elements of $[n]$ that are not blocked in $\mathsf v'$. In particular, $a_1=1$ (since $1$ cannot be blocked) and $a_r=n$ (since $n$ does not occur in $\mathsf v'$). The words in $\mathscr U_n(m)$ generated by $\mathsf v'$ are precisely the words of the form $\mathsf{w}na_\ell\mathsf v'$, where $1\leq \ell\leq r-1$ and $\mathsf{w}$ is a nondecreasing word of length $m-1$ over the alphabet $\{a_1,a_2,\ldots,a_\ell,n\}$. Let $\mathsf v$ be such a word generated by $\mathsf v'$. If the first letter in $\mathsf v$ is $n$, then $\mathsf w=nn\cdots n$, and the elements of $[n+1]$ that are not blocked in $\mathsf v$ are $a_1,\ldots,a_\ell,n,n+1$. In this case, the label of $\mathsf v$ is $\ell+2$. On the other hand, if the first letter in $\mathsf w$ is $a_j$ for some $1\leq j\leq \ell$, then the elements of $[n+1]$ that are not blocked in $\mathsf v$ are $a_1,\ldots,a_j,n+1$, so the label of $\mathsf v$ is $j+1$. The number of nondecreasing words of length $m-1$ over the alphabet $\{a_1,a_2,\ldots,a_\ell,n\}$ that start with $a_j$ (i.e., the number of choices for $\mathsf w$ once we have chosen $\mathsf v'$, $\ell$, and $j$) is $\binom{\ell-j+m-1}{m-2}$. (Indeed, this is the same as the number of multisets of cardinality $m-2$ with elements taken from the set $\{a_j,\ldots, a_\ell,n\}$.) 

Suppose $t\geq 2$. The previous paragraph shows that if $\mathsf v'\in\mathscr U_{n-1}(m)$ has label $r$, then the number of words $\mathsf v\in\mathscr U_n(m)$ that are generated by $\mathsf v'$ and have label $t$ is \[\sum_{\ell=t-1}^{r-1}\binom{\ell-t+m}{m-2}+\sum_{\ell=1}^{r-1}\delta_{\ell+2,t},\] where we are using the Kronecker $\delta$. By the hockey-stick identity, we have \[\sum_{\ell=t-1}^{r-1}\binom{\ell-t+m}{m-2}+\sum_{\ell=1}^{r-1}\delta_{\ell+2,t}=\binom{r+m-t}{m-1}-\delta_{t,2}.\] It follows that a generating tree describing the class $\bigcup_{n\geq 1}\mathscr U_n(m)$ is \[\text{Axiom: }(2)\qquad\text{Rule: }(r)\leadsto(2)^{\binom{r+m-2}{m-1}-1}(3)^{\binom{r+m-3}{m-1}}(4)^{\binom{r+m-4}{m-1}}\cdots(r+1)^{\binom{m-1}{m-1}}\text{ for all }r\geq 2,\] where we write $(t)^\alpha$ to denote $\alpha$ copies of $(t)$. This is also a generating tree for $\bigcup_{n\geq 1}\Tam_n'(m)$, the class of (nonempty) primitive $m$-ballot paths. 

\begin{theorem}\label{ThmMeet4}
For $n\geq 2$, let $\mathscr M_n(m)$ be the set of $m$-ballot paths $\mu\in\Tam_n(m)$ such that $\left|O_{\Pop_{\Tam_n(m)}}(\mu)\right|=m+n-1$. Then \[|\mathscr M_n(m)|=\frac{1}{n-1}\binom{(m+1)(n-2)+m-1}{n-2}.\] 
\end{theorem} 

\begin{proof}
Suppose $\mu\in\mathscr M_{n+1}(m)$. We must have $b_0(\mu)=0$ and $b_{m+1}(\mu)=1$. It follows from Theorem~\ref{ThmMeet3} that $b_j(\mu)=n+1$ for all $1\leq j\leq m$ and $b_{k(m+1)-1}(\mu)<n+1$ for all $2\leq k\leq n$. Let ${\bf y}(\mu)$ be the word $y_1(\mu)\cdots y_{(m+1)n}(\mu)$ over the alphabet $[n]$, where $y_i(\mu)=n+2-b_{m+1+i}(\mu)$. For example, if $m=n=2$ and ${\bf b}(\mu)=(0,3,3,1,3,2,2,3,3,3)$, then ${\bf y}(\mu)=122111$. Let $\mathscr Y_n(m)=\{{\bf y}(\mu):\mu\in\mathscr M_{n+1}(m)\}$. The map ${\bf y}$ is a bijection from $\mathscr M_{n+1}(m)$ to $\mathscr Y_n(m)$. By Definition~\ref{DefMeet2} and Theorem~\ref{ThmMeet3}, a word $\mathsf y=\mathsf y_1\cdots \mathsf y_{(m+1)n}$ over the alphabet $[n]$ is in $\mathscr Y_n(m)$ if and only if it satisfies the following conditions: 
\begin{itemize}
\item $\mathsf y_{k(m+1)}=n+1-k$ for all $1\leq k\leq n$; 
\item $1\leq \mathsf y_i\leq n+1-k$ whenever $1\leq k\leq n$ and $(k-1)(m+1)+1\leq i\leq k(m+1)$;
\item if $\mathsf y_i=n+1-k$, then $\mathsf y_j\geq n+1-k$ for all $i\leq j\leq k(m+1)$; 
\item $\mathsf y_{(k-1)(m+1)-1}(\mu)>1$ for all $2\leq k\leq n$.
\end{itemize}
Observe that the third condition could equivalently be replaced by the condition that $\mathsf y$ avoids the pattern $212$. 

For $\mathsf y=\mathsf y_1\cdots \mathsf y_{(m+1)n}\in\mathscr Y_n(m)$, let us say a letter $j$ is \dfn{blocked} in $\mathsf y$ if there exists a letter $i<j$ that appears to the left of an occurrence of $j$ in $\mathsf y$. Define the \dfn{label} of $\mathsf y$ to be the number of elements of $[n+1]$ that are not blocked in $\mathsf y$. 

If $n\geq 2$ and $\mathsf y\in\mathscr Y_n(m)$, then the word $\mathsf y'$ obtained by deleting the prefix of $\mathsf{y}$ of length $m+1$ is in $\mathscr Y_{n-1}(m)$. We say the word $\mathsf{y'}$ \dfn{generates} the word $\mathsf{y}$. For example, the word $344333122111\in\mathscr Y_4(2)$ is generated by the word $333122111\in\mathscr Y_3(2)$. Observe that since $\mathsf y$ avoids the pattern $212$, there cannot exist a letter that appears in the prefix of $\mathsf y$ of length $m+1$ and that is also blocked in $\mathsf y'$. 

Consider $\mathsf z'\in\mathscr Y_{n-1}(m)$, and let $a_1<\cdots<a_r$ be the elements of $[n]$ that are not blocked in $\mathsf z'$. In particular, $a_1=1$ (since $1$ cannot be blocked) and $a_r=n$ (since $n$ does not occur in $\mathsf z'$). The words generated by $\mathsf z'$ are precisely the words of the form $\mathsf{w}n\mathsf v'$, where $\mathsf{w}$ is a nondecreasing word of length $m$ over the alphabet $\{a_1,a_2,\ldots,a_r\}$ that does not end with the letter $1$. Let $\mathsf z$ be such a word generated by $\mathsf z'$, and let $a_j$ be the first letter of $\mathsf z$. The letters in $[n+1]$ that are not blocked in $\mathsf z$ are $a_1,\ldots,a_j,n+1$, so the label of $\mathsf z$ is $j+1$. If $a_j\neq 1$ (equivalently, $j\neq 1$), then the number of nondecreasing words of length $m$ over the alphabet $\{a_1,\ldots,a_r\}$ that start with $a_j$ and do not end in $1$ (i.e., the number of choices for $\mathsf w$ once we have chosen $\mathsf z'$ and $j$) is $\binom{r+m-j-1}{m-1}$. (Indeed, this is the same as the number of multisets of cardinality $m-1$ with elements taken from the set $\{a_j,\ldots, a_r\}$.) On the other hand, if $a_j=1$ (equivalently, $j=1$), then the number of nondecreasing words over $\{a_1,\ldots,a_r\}$ that start with $a_j$ and do not end in $1$ is $\binom{r+m-j-1}{m-1}-1=\binom{r+m-2}{m-1}-1$ (we must exclude the word $11\cdots 1$). 

It follows from the previous paragraph that a generating tree describing the class $\bigcup_{n\geq 1}\mathscr Y_n(m)$ is \[\text{Axiom: }(2)\qquad\text{Rule: }(r)\leadsto(2)^{\binom{r+m-2}{m-1}-1}(3)^{\binom{r+m-3}{m-1}}(4)^{\binom{r+m-4}{m-1}}\cdots(r+1)^{\binom{m-1}{m-1}}\text{ for all }r\geq 2.\] This is exactly the same generating tree that we found above for the class $\bigcup_{n\geq 1}\mathscr U_n(m)$. Therefore, \[|\mathscr M_{n+1}(m)|=|\mathscr Y_n(m)|=|\mathscr U_n(m)|=|\!\Tam_n'(m)|.\] The desired result is now a consequence of \eqref{EqMeet3}. 
\end{proof}

By specializing Theorems~\ref{ThmMeet3} and \ref{ThmMeet4} to the Tamari lattices $\Tam_n$, we obtain the following corollary, which is also a consequence of Proposition~\ref{PropMeet1}. 
 
\begin{corollary}
Suppose $n\geq 2$. We have \[\max_{\mu\in\Tam_n}\left|O_{\Pop_{\Tam_n}}(\mu)\right|=n.\] The number of Dyck paths $\mu\in\Tam_n$ such that $\left|O_{\Pop_{\Tam_n}}(\mu)\right|=n$ is the Catalan number $C_{n-2}$. 
\end{corollary}

\subsection{$1$-$\Pop$-Sortable Lattice Paths}\label{Subsec1Sortable}

We now turn back to $\nu$-Tamari lattices for general $\nu$. Suppose $\nu$ starts at $(0,0)$ and ends at $(\ell-n,n)$. Note that there are unique nonnegative integers $\gamma_0,\ldots,\gamma_n$ such that $\nu=\text{E}^{\gamma_0}\text{NE}^{\gamma_1}\cdots\text{NE}^{\gamma_n}$. Recall that a lattice path $\mu\in\Tam(\nu)$ is called \emph{$1$-$\Pop$-sortable} if $\Pop(\mu)=\nu$.  

\begin{theorem}\label{ThmMeet5}
Let $\nu=\E^{\gamma_0}\!\N\!\E^{\gamma_1}\cdots\N\!\E^{\gamma_n}$ be a lattice path, where $\gamma_0,\ldots,\gamma_n$ are nonnegative integers. The number of $1$-$\Pop$-sortable lattice paths in $\Tam(\nu)$ is $2^{|\mathscr A(\nu)|}$, where $\mathscr A(\nu)$ is the set of indices $k\in\{0,\ldots,n-1\}$ such that $\gamma_k\geq 1$. 
\end{theorem}

\begin{proof}
Note that $\ell=\gamma_0+\cdots+\gamma_n+n$. Let $f_0,\ldots,f_n$ be the fixed positions of $\nu$. Then $f_k=\gamma_0+\cdots+\gamma_k+k$ for $0\leq k\leq n$, and $f_{-1}=-1$ by convention. Suppose $\mu$ is $1$-$\Pop$-sortable, and let $\alpha(\mu)$ be the set of indices $k\in\{0,\ldots,n\}$ such that $b_{f_{k-1}+1}(\mu)>k$. Notice that if $\gamma_k=0$, then $k\not\in\alpha(\mu)$ because $b_{f_{k-1}+1}(\mu)=b_{f_k}(\mu)=k$. Also, $n\not\in\alpha(\mu)$ since $b_{f_{n-1}+1}(\mu)=n$. Thus, $\alpha$ is a map from the set $\Pop^{-1}(\nu)$ of $1$-$\Pop$-sortable elements of $\Tam(\nu)$ to the set $2^{\mathscr A(\nu)}$ of all subsets of $\mathscr A(\nu)$. We claim that this map is a bijection; to prove this, we will exhibit its inverse $\beta:2^{\mathscr A(\nu)}\to\Pop^{-1}(\nu)$. 

Let $Q\subseteq\mathscr A(\nu)$. There are unique integers $q_1<r_1<q_2<r_2<\cdots<q_t<r_t$ such that $Q=[q_1,r_1-1]\cup\cdots\cup[q_t,r_t-1]$, where we write $[a,b]$ for the set $\{a,a+1,\ldots,b\}$. For each $i\in[t]$ and each $k\in[q_i,r_i-1]$, let $c_{f_{k-1}+1}=r_i$. For each $j\in\{0,\ldots,\ell\}\setminus\{f_{k-1}+1:k\in Q\}$, let $c_j=b_j(\nu)$. (See Example~\ref{ExamMeet3}.) We readily check that $(c_0,\ldots,c_\ell)$ is a $\nu$-bracket vector. By Theorem~\ref{ThmBracket}, there exists a unique lattice path $\beta(Q)\in\Tam(\nu)$ such that ${\bf b}(\beta(Q))=(c_0,\ldots,c_\ell)$. It is immediate from Proposition~\ref{LemMeet2} that $\Pop(\beta(Q))=\nu$, so $\beta(Q)\in\Pop^{-1}(\nu)$. It is also a direct consequence of our construction that $\alpha(\beta(Q))=Q$. To complete the proof, we must show that $\beta(\alpha(\mu))=\mu$ for all $\mu\in\Pop^{-1}(\nu)$. 

Choose $\mu\in\Pop^{-1}(\nu)$. Write $\alpha(\mu)=[q_1,r_1-1]\cup\cdots\cup[q_t,r_t-1]$, where $q_1<r_1<q_2<r_2<\cdots<q_t<r_t$. Let $(c_0,\ldots,c_\ell)={\bf b}(\beta(\alpha(\mu)))$. Consider some $j\in\{0,\ldots,\ell\}$ that is not of the form $f_{k-1}+1$ for $k\in\alpha(\mu)$. If $j=f_{k-1}+1$ for some $k\not\in\alpha(\mu)$, then it follows from the definition of $\alpha$ that $b_j(\mu)=k=b_j(\nu)=c_j$. On the other hand, if $j\not\in\{f_{-1}+1,f_0+1,\ldots,f_{n-1}+1\}$, then we can use Corollary~\ref{CorMeet1} to see that $b_{j-1}(\Pop(\mu))\geq b_j(\mu)\geq b_j(\nu)$. In this case, we can use the assumption that $\mu\in\Pop^{-1}(\nu)$ to see that $b_{j-1}(\Pop(\mu))=b_{j-1}(\nu)=b_j(\nu)$, so $b_j(\mu)=b_j(\nu)=c_j$. Now choose $i\in[t]$. We have $c_{f_{k-1}+1}=r_i$ for all $k\in[q_i,r_i-1]$, so we want to prove that $b_{f_{k-1}+1}(\mu)=r_i$ for all $k\in[q_i,r_i-1]$. We proceed by backward induction on $k$. Suppose we have already proven that $b_{f_{\kappa-1}+1}(\mu)=r_i$ for all $\kappa\in[k+1,r_i-1]$. Let $a=b_{f_{k-1}+1}(\mu)$. Since $k\in\alpha(\mu)$, we must have $a>k$. If $k+1\leq a\leq r_i-1$, then our induction hypothesis guarantees that $b_{f_{a-1}+1}(\mu)=r_i>a$. However, this is impossible because condition \eqref{Item3} in Definition~\ref{DefMeet2} forces $b_j\leq a$ for all $f_{k-1}+2\leq j\leq f_a$. We deduce that $a\geq r_i$. 

Suppose $a>r_i$. It follows from the discussion above that $b_{f_{k-1}+2}(\mu)=k$. Because $k\in\alpha(\mu)$, we have $b_{f_{k-1}+1}(\mu)>k=b_{f_{k-1}+2}(\mu)$. Invoking Lemma~\ref{LemMeet3}, we find that $f_{k-1}+1\in\Delta(\mu)$. Recall that $\eta_{f_{k-1}+1}(\mu)$ is defined to be the largest element $h$ of $\{b_{f_{k-1}+1}(\nu),\ldots,$ $b_{f_{k-1}+1}(\mu)-1\}$ such that $b_j(\mu)\leq h$ for all $f_{k-1}+2\leq j\leq f_{h}$. It follows from the above discussion and our induction hypothesis that $b_j(\mu)\leq r_i$ for all $f_{k-1}+2\leq j\leq f_{r_i}$. Indeed, for each such $j$, we either have $b_j(\mu)=b_j(\nu)$ or $b_j(\mu)=r_i$. This shows that $\eta_{f_{k-1}+1}(\mu)\geq r_i$. It follows from Proposition~\ref{LemMeet2} that $b_{f_{k-1}+1}(\Pop(\mu))=\eta_{f_{k-1}+1}(\mu)\geq r_i>k=b_{f_{k-1}+1}(\nu)$, contradicting the assumption that $\Pop(\mu)=\nu$. From this contradiction, we deduce that $a=r_i$, as desired. This completes the induction and, therefore, the proof that $\beta(\alpha(\mu))=\mu$.  
\end{proof}

\begin{example}\label{ExamMeet3}
Let $\nu=\text{NE}^2\text{NENE}^2\text{NNENE}^3\text{NE}^2$. Then $\ell=18$, $n=7$, and the values of $\gamma_0,\ldots, \gamma_7$ are $0,2,1,2,0,1,3,2$. The fixed positions $f_0,\ldots,f_7$ are $0,3,5,8,9,11,15,18$. We have $\mathscr A(\nu)=\{1,2,3,5,6\}$. Suppose we choose the subset $Q=\{1,2,3,6\}$ of $\mathscr A(\nu)$. Then $Q=[1,4-1]\cup[6,7-1]$, so $q_1=1$, $r_1=4$, $q_2=6$, and $r_2=7$. Since $1\in[q_1,r_1-1]$, we define $c_{f_{1-1}+1}=c_1$ to be $r_1=4$. Similarly, we define $c_{f_{2-1}+1}=c_4=r_1=4$, $c_{f_{3-1}+1}=c_6=r_1=4$, and $c_{f_{6-1}+1}=c_{12}=r_2=7$. For $j\in\{0,\ldots,18\}\setminus\{1,4,6,12\}$, we put $c_j=b_j(\nu)$. The resulting vector $(c_0,\ldots,c_{18})$ is \[(0,4,1,1,4,2,4,3,3,4,5,5,7,6,6,6,7,7,7).\] Thus, $\beta(Q)=\text{NENNENE}^3\text{NENE}^2\text{NE}^3\in\Pop^{-1}(\nu)$.  
\end{example} 

\begin{corollary}
For $m,n\geq 1$, the number of $1$-$\Pop$-sortable $m$-ballot paths in $\Tam_n(m)$ is $2^{n-1}$. 
\end{corollary}

\begin{proof}
Set $\nu=(\text{NE}^m)^n$ in Theorem~\ref{ThmMeet5}, and note that $\mathscr A(\nu)=\{1,\ldots,n-1\}$. 
\end{proof}

\subsection{$2$-$\Pop$-Sortable Lattice Paths}\label{Subsec2Sortable}

In this subsection, we work with $\nu$-Tamari lattices for varying $\nu$. These various $\nu$ will have different lengths; therefore, the notation ${\bf b}(\mu)$ will be unambiguous because there will only be one $\nu$ under consideration such that $\mu\in\Tam(\nu)$. 

Let $\nu$ be a lattice path that starts at $(0,0)$, ends at $(\ell-n,n)$, and has fixed positions $f_0,\ldots,f_n$, where $n\geq 2$. Let $\nu^{\#}$ be the lattice path obtained from $\nu$ by deleting the first $f_0+1$ steps from $\nu$. In other words, if $\nu=\text{E}^{\gamma_0}\text{NE}^{\gamma_1}\cdots\text{NE}^{\gamma_n}$, then $\nu^{\#}=\text{E}^{\gamma_1}\text{NE}^{\gamma_2}\cdots\text{NE}^{\gamma_n}$. The $\nu^{\#}$-bracket vector ${\bf b}(\nu^{\#})$ is obtained by deleting the first $f_0+1$ entries of the $\nu$-bracket vector ${\bf b}(\nu)$ and decreasing each of the remaining entries by $1$. More generally, if $\vec{\mathsf b}$ is a $\nu$-bracket vector, then we obtain a $\nu^{\#}$-bracket vector $\vec{\mathsf b}^{\#}$ by deleting the first $f_0+1$ entries of $\vec{\mathsf b}$ and decreasing each of the remaining entries by $1$. For any $\mu\in\Tam(\nu)$, let $\mu^{\#}$ be the unique lattice path in $\Tam(\nu^{\#})$ such that ${\bf b}(\mu^\#)={\bf b}(\mu)^\#$. 

Proposition~\ref{LemMeet2} tells us that $b_i(\Pop(\mu))$ depends only on the entries $b_j(\mu)$ for $j\geq i$. It follows that if $\mu\in\Tam(\nu)$ is $t$-$\Pop$-sortable, then so is the lattice path $\mu^\#\in\Tam(\nu^\#)$. This yields a recursive method to generate and enumerate the $t$-$\Pop$-sortable elements of $\Tam(\nu)$. In this subsection, we detail this method for $t=2$. 

Suppose we are given a $2$-$\Pop$-sortable lattice path $\mu_0\in\Tam(\nu^\#)$. We want to describe all the ways to construct a $2$-$\Pop$-sortable lattice path $\mu\in\Tam(\nu)$ such that $\mu^\#=\mu_0$. Let $r_i=b_i(\mu_0)+1$ so that ${\bf b}(\mu)=(b_0(\mu),b_1(\mu),\ldots,b_{f_0}(\mu),r_0,r_1,\ldots,r_{\ell-f_0-1})$. Note that $1\leq r_i\leq n$ for all $0\leq i\leq \ell-f_0-1$. We need to describe the possible choices for $b_0(\mu),b_1(\mu),\ldots,b_{f_0}(\mu)$. We must put $b_{f_0}(\mu)=0$. This tells us that there is only one choice of $\mu$ if $f_0=0$. Thus, we may assume in what follows that $f_0\geq 1$. Corollary~\ref{CorMeet1} informs us that if $f_0\geq 2$, then $b_0(\Pop^2(\mu))\geq b_1(\Pop(\mu))\geq b_2(\mu)$. Since we are trying to construct $\mu$ so that $\Pop^2(\mu)=\nu$, we must have $b_0(\Pop^2(\mu))=b_0(\nu)=0$. Therefore, $b_j(\mu)=0$ for all $2\leq j\leq f_0$. This shows that we only need to describe the possible choices for $b_0(\mu)$ and $b_1(\mu)$. There are five cases to consider. 

\medskip 

\noindent {\bf Case 1:} Assume $f_0\geq 2$ and either $r_0=r_1<n$ or $r_0<r_1$. Observe that the inequality $r_0<r_1$ can only hold if $f_1=f_0+1$; in this case, $r_0=b_{f_1}(\mu)=1$. In either case, we have $b_{f_0+1}(\Pop(\mu))=r_0$, where we are using Corollary~\ref{CorMeet1} when $r_0=r_1<n$. Let $d=r_{f_{r_0}-f_0}=b_{f_{r_0}+1}(\mu)$. We claim that the possible choices for the pair $(b_0(\mu),b_1(\mu))$ are \[(0,0),\quad (r_0,0), \quad (d,0),\quad (r_0,r_0),\quad (d,r_0).\] Let us first check that each of these pairs is allowed. To do this, we need to verify that the resulting pair $(b_0(\Pop^2(\mu)),b_1(\Pop^2(\mu)))$ is $(0,0)$. Using Proposition~\ref{LemMeet2}, one can check that the pairs $(b_0(\Pop(\mu)),b_1(\Pop(\mu)))$ corresponding to the five pairs listed above are (in the corresponding order)
\[(0,0),\quad (0,0), \quad (r_0,0),\quad (r_0,0),\quad (r_0,0).\] Since $b_{f_0+1}(\Pop(\mu))=r_0$, we can use Proposition~\ref{LemMeet2} again to see that $(b_0(\Pop^2(\mu)),b_1(\Pop^2(\mu)))=(0,0)$ in each case. See Example~\ref{ExamMeet2}.  

We now want to show that these are the only possible choices. If $1\leq b_1(\mu)\leq r_0-1$, then $2\leq f_0+1\leq f_{b_1(\mu)}$ and $b_{f_0+1}(\mu)=r_0>b_1(\mu)$, contradicting condition \eqref{Item3} in Definition~\ref{DefMeet2}. If $r_0+1\leq b_1(\mu)\leq d-1$, then $2\leq f_{r_0}+1\leq f_{b_1(\mu)}$ and $b_{f_{r_0}+1}(\mu)=d>b_1(\mu)$, again contradicting Definition~\ref{DefMeet2}. This shows that $b_1(\mu)$ cannot belong to the set $\{1,\ldots,d-1\}\setminus\{r_0\}$. The exact same argument shows that $b_0(\mu)$ also cannot belong to $\{1,\ldots,d-1\}\setminus\{r_0\}$. Suppose $b_0(\mu)>d$ or $(b_0(\mu),b_1(\mu))=(d,d)$. If $b_0(\mu)>d$, then it is straightforward to check that $\eta_0(\mu)\geq d$, so $b_0(\Pop(\mu))\geq d$ by Proposition~\ref{LemMeet2}. In the second case, we have $b_0(\Pop(\mu))\geq d$ by Corollary~\ref{CorMeet1}. Since $b_{f_0+1}(\Pop(\mu))=r_0$, one can check that $\eta_0(\Pop(\mu))\geq r_0>0$, which is impossible by Proposition~\ref{LemMeet2} since we need to have $b_0(\Pop^2(\mu))=0$. This shows that $b_0(\mu)\leq d$ and $(b_0(\mu),b_1(\mu))\neq(d,d)$. Since $b_1(\mu)\leq b_0(\mu)\leq d$, we obtain the desired claim. 

\medskip 

\noindent {\bf Case 2:} Assume $f_0=1$ and either $r_0=r_1<n$ or $r_0<r_1$. This case is exactly the same as Case~1, except that we must have $b_1(\mu)=0$. Hence, the possible choices for the pair $(b_0(\mu),b_1(\mu))$ are \[(0,0),\quad (r_0,0), \quad (d,0),\] where $d=b_{f_{r_0}+1}(\mu)$. 

\medskip

\noindent {\bf Case 3:} Assume $f_0\geq 2$ and $r_0=r_1=n$. In this case, the possible choices for the pair $(b_0(\mu),b_1(\mu))$ are \[(0,0),\quad (n,0), \quad (n,n).\] The proof is exactly the same as in Case~1, except that we no longer have the entry $d$. 

\medskip 

\noindent {\bf Case 4:} Assume $f_0=1$ and $r_0=r_1=n$. In this case, the possible choices for $(b_0(\mu),b_1(\mu))$ are \[(0,0),\quad (n,0).\] Indeed, this case is exactly the same as Case~3, except that we must have $b_1(\mu)=0$. 

\medskip 

\noindent {\bf Case 5:} Assume $r_0>r_1$. This implies that $f_1\geq f_0+2$. By Lemma~\ref{LemMeet3}, $f_0+1\in\Delta(\mu)$, so Proposition~\ref{LemMeet2} tells us that 
\begin{equation}\label{EqMeet6}
b_{f_0+1}(\Pop(\mu))<b_{f_0+1}(\mu)=r_0.
\end{equation} We claim that the possible choices for $(b_0(\mu),b_1(\mu))$ are \[(0,0),\quad (r_0,0).\] Let us first check that each of these pairs is allowed. If $(b_0(\mu),b_1(\mu))=(0,0)$, then certainly $(b_0(\Pop^2(\mu)),b_1(\Pop^2(\mu)))=(0,0)$. Now suppose $(b_0(\mu),b_1(\mu))=(r_0,0)$. Then $0\in\Delta(\mu)$ by Lemma~\ref{LemMeet3}, so $\eta_0(\mu)<b_0(\mu)=r_0$. If $\eta_0(\mu)\geq 1$, then $1\leq f_0+1\leq f_{\eta_0(\mu)}$ and $b_{f_0+1}(\mu)=r_0>\eta_0(\mu)$, contradicting the definition of $\eta_0(\mu)$. Therefore, $\eta_0(\mu)=0$. Proposition~\ref{LemMeet2} tells us that $(b_0(\Pop(\mu)),b_1(\Pop(\mu)))=(0,0)$, so $(b_0(\Pop^2(\mu)),b_1(\Pop^2(\mu)))=(0,0)$. 

We want to show that these are the only possible choices. If $1\leq b_1(\mu)\leq r_0-1$, then $2\leq f_0+1\leq f_{b_1(\mu)}$ and $b_{f_0+1}(\mu)=r_0>b_1(\mu)$, contradicting condition \eqref{Item3} in Definition~\ref{DefMeet2}. The same argument shows that we cannot have $1\leq b_0(\mu)\leq r_0-1$. Now suppose, by way of contradiction, that $b_0(\Pop(\mu))\geq r_0$. Let $h=b_{f_0+1}(\Pop(\mu))$. Since $\mu$ is supposed to be $2$-$\Pop$-sortable, it follows from Corollary~\ref{CorMeet1} that $0=b_0(\Pop^2(\mu))\geq b_1(\Pop(\mu))$. Hence, $b_j(\Pop(\mu))=0\leq h$ for all $2\leq j\leq f_0$. Because ${\bf b}(\Pop(\mu))$ is a $\nu$-bracket vector, it follows from Definition~\ref{DefMeet2} that $b_j(\Pop(\mu))\leq h$ for all $f_0+1\leq j\leq f_h$. Hence, $b_j(\Pop(\mu))\leq h$ for all $2\leq j\leq f_h$. We noted in \eqref{EqMeet6} that $h<b_{f_0+1}(\mu)=r_0$. Therefore, it follows from the definition of $\eta_0(\Pop(\mu))$ that $\eta_0(\Pop(\mu))\geq h$. By Proposition~\ref{LemMeet2}, $0=b_0(\Pop^2(\mu))=\eta_0(\Pop(\mu))\geq h\geq 1$, which is impossible. From this contradiction, we deduce that $b_0(\Pop(\mu))<r_0$. It follows from Corollary~\ref{CorMeet1} that $b_1(\mu)<r_0$. We already showed that $b_1(\mu)$ cannot belong to the set $\{1,\ldots, r_0-1\}$, so $b_1(\mu)=0$. This implies that $b_j(\mu)=0<r_0$ for all $1\leq j\leq f_0$. We also have $b_j(\mu)\leq r_0$ for all $f_0+1\leq j\leq f_{r_0}$ since ${\bf b}(\mu)$ is a $\nu$-bracket vector and $b_{f_0+1}(\mu)=r_0$. If $b_0(\mu)>r_0$, then we can use Proposition~\ref{LemMeet2} to see that $b_0(\Pop(\mu))=\eta_0(\mu)\geq r_0$. This is a contradiction, so we deduce that $b_0(\mu)\leq r_0$. Hence, $(b_0(\mu),b_1(\mu))$ is either $(0,0)$ or $(r_0,0)$.

\begin{example}\label{ExamMeet2}
Suppose ${\bf b}(\nu)=(0,0,0,1,1,1,2,2,3,3,4,5)$. Then ${\bf b}(\nu^\#)=(0,0,0,1,1,2,2,3,4)$. Suppose $\mu^\#$ is chosen so that ${\bf b}(\mu^\#)=(1,1,0,1,1,3,2,3,4)$; one can check that $\mu^\#$ is indeed a $2$-$\Pop$-sortable element of $\Tam(\nu^\#)$. Then $(r_0,\ldots,r_8)=(2,2,1,2,2,4,3,4,5)$. The lattice path $\mu$ that we construct will be such that ${\bf b}(\mu)=(b_0(\mu),b_1(\mu),0,2,2,1,2,2,4,3,4,5)$. Since $r_0=r_1=2<5=n$, we are in Case 1. We have $d=b_{f_{r_0}+1}(\mu)=b_{f_2+1}(\mu)=b_8(\mu)=4$. The discussion above tells us that the possible choices for $(b_0(\mu),b_1(\mu))$ are $(0,0), (2,0), (4,0), (2,2), (4,2)$. Let us illustrate why $(2,2)$ and $(4,2)$ are allowable pairs. 

Suppose we choose $(b_0(\mu),b_1(\mu))=(2,2)$. Then ${\bf b}(\mu)=(2,2,0,2,2,1,2,2,4,3,4,5)$. We have ${\bf b}(\Pop(\mu))=(2,0,0,2,1,1,2,2,3,3,4,5)$ and ${\bf b}(\Pop^2(\mu))=(0,0,0,1,1,1,2,2,3,3,4,5)={\bf b}(\nu)$.

Suppose we choose $(b_0(\mu),b_1(\mu))=(4,2)$ instead. Then ${\bf b}(\mu)=(4,2,0,2,2,1,2,2,4,3,4,5)$, ${\bf b}(\Pop(\mu))=(2,0,0,2,1,1,2,2,3,3,4,5)$, and ${\bf b}(\Pop^2(\mu))=(0,0,0,1,1,1,2,2,3,3,4,5)={\bf b}(\nu)$.

Notice that we cannot choose $(b_0(\mu),b_1(\mu))=(3,2)$ since $(3,2,0,2,2,1,2,2,4,3,4,5)$ is not a $\nu$-bracket vector. Suppose we were to choose $(b_0(\mu),b_1(\mu))=(4,4)$. Then we would have ${\bf b}(\mu)=(4,4,0,2,2,1,2,2,4,3,4,5)$, ${\bf b}(\Pop(\mu))=(4,0,0,2,1,1,2,2,3,3,4,5)$ and ${\bf b}(\Pop^2(\mu))=(3,0,0,1,1,1,2,2,3,3,4,5)\neq{\bf b}(\nu)$, meaning $\mu$ would \emph{not} be $2$-$\Pop$-sortable. 
\end{example}

In the remainder of this section, we apply the previous discussion to enumerate $2$-$\Pop$-sortable lattice paths in $\nu$-Tamari lattices for some more specific (but still quite general) lattice paths $\nu$. 

Suppose $\nu=\text{E}^{\alpha_0}\text{N}^{\beta_0}\text{E}^{\alpha_1}\text{N}^{\beta_1}\cdots\text{E}^{\alpha_{q-1}}\text{N}^{\beta_{q-1}}\text{E}^{\alpha_q}$, where $\alpha_0,\ldots,\alpha_q,\beta_0,\ldots,\beta_{q-1}$ are nonnegative integers and $\alpha_1,\ldots,\alpha_{q-1},\beta_0,\ldots,\beta_{q-1}\geq 1$. Define $\theta(\nu)$ to be the number of indices $i\in\{0,\ldots,q-1\}$ such that $\alpha_i=1$. Let $\chi(\nu)$ be the number of indices $i\in\{0,\ldots,q-1\}$ such that $\alpha_i\geq 2$.

\begin{theorem}\label{ThmMeet6}
Let $\nu=\E^{\alpha_0}\!\N^{\beta_0}\!\E^{\alpha_1}\!\N^{\beta_1}\cdots\E^{\alpha_{q-1}}\!\N^{\beta_{q-1}}\!\E^{\alpha_{q}}$, where $\alpha_0,\ldots,\alpha_{q},\beta_0,\ldots,\beta_{q-1}$ are nonnegative integers such that $\alpha_1,\ldots,\alpha_{q-1}\geq 1$ and $\beta_0,\ldots,\beta_{q-1}\geq 2$. The number of $2$-$\Pop$-sortable lattice paths in $\Tam(\nu)$ is \[3^{\theta(\nu)}5^{\chi(\nu)}.\] 
\end{theorem}

\begin{proof}
Suppose $\nu$ starts at $(0,0)$ and ends at $(\ell-n,n)$. If $\ell=0$ or $q=0$, then there is only $1$ element of $\Tam(\nu)$. In either of these cases, the desired result holds because $\theta(\nu)=\chi(\nu)=0$. We now assume $q\geq 1$ and $\ell\geq 1$ and proceed by induction on $\ell$. 

Preserving the notation from above, we have $\nu^\#=\text{N}^{\beta_0-1}\text{E}^{\alpha_1}\text{N}^{\beta_1}\cdots\text{E}^{\alpha_{q-1}}\text{N}^{\beta_{q-1}}\text{E}^{\alpha_q}$. If $\alpha_0=0$, then $\theta(\nu)=\theta(\nu^\#)$, $\chi(\nu)=\chi(\nu^\#)$, and the lattice $\Tam(\nu)$ is isomorphic to the lattice $\Tam(\nu^\#)$ (this is immediate from Pr\'eville-Ratelle and Viennot's original definition of these lattices that we presented earlier). In this case, \[|\!\Pop^{-2}(\nu)|=|\!\Pop^{-2}(\nu^\#)|=3^{\theta(\nu^\#)}5^{\chi(\nu^\#)}=3^{\theta(\nu)}5^{\chi(\nu)},\] where the identity $|\!\Pop^{-2}(\nu^\#)|=3^{\theta(\nu^\#)}5^{\chi(\nu^\#)}$ follows by induction on $\ell$. This proves the desired result when $\alpha_0=0$, so we may assume in what follows that $\alpha_0\geq 1$. 

Let $\mu_0$ be a $2$-$\Pop$-sortable element of $\Tam(\nu^\#)$. Let $f_0,\ldots,f_n$ be the fixed positions of $\nu$, and let $r_i=b_i(\mu_0)+1$. Since $\beta_0\geq 2$, the first step in $\nu^\#$ is a north step. This implies that $b_0(\mu_0)=0$, so $r_0=1$. The hypothesis that $\beta_0\geq 2$ also guarantees that $n\geq 2$. Therefore, we either have $r_0=r_1<n$ or $r_0<r_1$. We now consider two cases based on whether $\alpha_0=1$ or $\alpha_0\geq 2$. 

Suppose $\alpha_0=1$. Then $\theta(\nu)=\theta(\nu^\#)+1$ and $\chi(\nu)=\chi(\nu^\#)$. Furthermore, $f_0=1$, so we are in Case~2 from above. In this case, there are $3$ different $2$-$\Pop$-sortable lattice paths $\mu\in\Tam(\nu)$ such that $\mu^\#=\mu_0$. As this is true for all $2$-$\Pop$-sortable lattice paths $\mu_0\in\Tam(\nu^\#)$, we find that \[|\!\Pop^{-2}(\nu)|=3|\!\Pop^{-2}(\nu^\#)|=3^{\theta(\nu^\#)+1}5^{\chi(\nu^\#)}=3^{\theta(\nu)}5^{\chi(\nu)}.\] 

Now suppose $\alpha_0\geq 2$. Then $\theta(\nu)=\theta(\nu^\#)$ and $\chi(\nu)=\chi(\nu^\#)+1$. Furthermore, $f_0\geq 2$, so we are in Case~1 from above. In this case, there are $5$ different $2$-$\Pop$-sortable lattice paths $\mu\in\Tam(\nu)$ such that $\mu^\#=\mu_0$. As this is true for all $2$-$\Pop$-sortable lattice paths $\mu_0\in\Tam(\nu^\#)$, we find that \[|\!\Pop^{-2}(\nu)|=5|\!\Pop^{-2}(\nu^\#)|=3^{\theta(\nu^\#)}5^{\chi(\nu^\#)+1}=3^{\theta(\nu)}5^{\chi(\nu)}.\qedhere\] 
\end{proof}

Our final results enumerate $2$-$\Pop$-sortable elements of $m$-Tamari lattices, which are lattices to which Theorem~\ref{ThmMeet6} does not apply. There are two cases to consider: the Tamari lattices $\Tam_n$ and the $m$-Tamari lattices $\Tam_n(m)$ for $m\geq 2$. Indeed, it is not difficult to check that for $m\geq 2$, the number of $2$-$\Pop$-sortable $m$-ballot paths in $\Tam_n(m)$ is independent of $m$ (we will see this in the proof of Theorem~\ref{ThmMeet7}). In fact, one can show that for $m\geq t$, the number of $t$-$\Pop$-sortable $m$-ballot paths in $\Tam_n(m)$ is independent of $m$. 

We first handle the Tamari lattices. Let $P_n$ denote the $n^\text{th}$ Pell number. These numbers, which form the sequence A000129 in \cite{OEIS}, satisfy $P_1=1$, $P_2=2$, and $P_{n+1}=2P_n+P_{n-1}$ for $n\geq 2$. 

\begin{theorem}
For each $n\geq 1$, the number of $2$-$\Pop$-sortable Dyck paths in the Tamari lattice $\Tam_n$ is the Pell number $P_n$. 
\end{theorem}

\begin{proof}
Let $a(n)$ be the number of $2$-$\Pop$-sortable Dyck paths in $\Tam_n$. We trivially have $a(1)=1=P_1$ and $a(2)=2=P_2$. Now suppose $n\geq 2$. We find it convenient to work with $\Tam(\nu)$, where $\nu=(\text{EN})^{n}\text{E}$; note that this lattice is isomorphic to $\Tam_{n+1}$. We have $\nu^\#=(\text{EN})^{n-1}\text{E}$. Let $\mu_0$ be a $2$-$\Pop$-sortable lattice path in $\Tam(\nu^\#)$. We want to determine the number of lattice paths $\mu\in\Tam(\nu)$ such that $\mu^\#=\mu_0$. Let $f_0,\ldots,f_n$ be the fixed positions of $\nu$, and let $r_i=b_i(\mu_0)+1$. Observe that $f_0=1$ and $r_1=1$. 

If $r_0=r_1=1$, then we are in Case~2 from above, so there are $3$ different choices for $\mu$. If $r_0\neq 1$, then we are in Case~5 from above, so there are $2$ different choices for $\mu$. We have $r_0=r_1=1$ if and only if $\mu_0=\text{EN}\mu^*$ for some $2$-$\Pop$-sortable lattice path $\mu^*\in\Tam((\text{EN})^{n-2}\text{E})$; the number of such lattice paths is $a(n-1)$ because $\Tam((\text{EN})^{n-2}\text{E})$ is isomorphic to $\Tam_{n-1}$. Consequently, the number of choices for $\mu_0$ such that $r_0\neq 1$ is $a(n)-a(n-1)$. This shows that $a(n+1)=3a(n-1)+2(a(n)-a(n-1))=2a(n)+a(n-1)$. Hence, the sequence $a(1),a(2),\ldots$ satisfies the same initial conditions and recurrence relation as the sequence of Pell numbers.  
\end{proof}

The numbers $g(1),g(2),\ldots$ in the following theorem form sequence A006190 in \cite{OEIS}. 

\begin{theorem}\label{ThmMeet7}
Fix $m\geq 2$. For $n\geq 1$, let $g(n)$ be the number of $2$-$\Pop$-sortable $m$-ballot paths in $\Tam_n(m)$. Then \[\sum_{n\geq 1}g(n)z^n=\frac{z}{1-3z-z^2}.\]
\end{theorem}

\begin{proof}
The lattice $\Tam_1(m)$ has only $1$ element, so $g(1)=1$. Following the discussion above, we find that an $m$-ballot path $\mu\in\Tam_2(m)$ is $2$-$\Pop$-sortable if and only if ${\bf b}(\mu)$ is of the form $(0,x,y,1,1,\ldots,1,2,2,\ldots,2)$, where $(x,y)$ is $(2,2)$, $(2,1)$, or $(1,1)$. Thus, $g(2)=3$. Now suppose $n\geq 2$. We find it convenient to work with the lattice paths $\nu_k=(\text{E}^m\text{N})^{k}\text{E}^m$; notice that the lattices $\Tam(\nu_k)$ and $\Tam_{k+1}(m)$ are isomorphic. Let $A_k$ be the set of $2$-$\Pop$-sortable lattice paths $\mu\in\Tam(\nu_k)$ such that either $b_0(\mu)+1=b_1(\mu)+1<k+1$ or $b_0(\mu)+1<b_1(\mu)+1$. Let $A_k'$ be the set of $2$-$\Pop$-sortable lattice paths $\mu\in\Tam(\nu_k)$ such that $b_0(\mu)+1=b_1(\mu)+1=k+1$. Consider $\nu=\nu_n$. We have $\nu_n^\#=(\text{E}^m\text{N})^{n-1}\text{E}^m=\nu_{n-1}$. 
For each $2$-$\Pop$-sortable lattice path $\mu_0\in\Tam(\nu_n^\#)$, we want to determine the number of lattice paths $\mu\in\Tam(\nu_n)$ such that $\mu^\#=\mu_0$. Let $f_0,\ldots,f_n$ be the fixed positions of $\nu_n$, and let $r_i=b_i(\mu_0)+1$. Observe that $f_0=m\geq 2$. This means we must be in Case~1, Case~3, or Case~5 from above. 

By inspecting the possibilities for the pair $(b_0(\mu),b_1(\mu))$ in each of Cases~1,~3, and~5, we find that the only way we can have $\mu\in A_n'$ is if $\mu_0\in A_{n-1}'$; furthermore, if $\mu_0\in A_{n-1}'$, then there is a unique $\mu\in A_n'$ such that $\mu^\#=\mu_0$. This shows that $|A_n'|=|A_{n-1}'|$. The only element of $A_1'$ is the lattice path $\text{E}^{m-2}\text{NE}^{m+2}$ with $\nu_1$-bracket vector $(1,1,0,0,\ldots,0,1,1,\ldots,1)$, so $|A_1'|=1$. By induction, $|A_k'|=1$ for all $k\geq 1$. 

A similar inspection shows that the only way we can have $\mu\in A_n$ is if either $\mu_0\in\Tam(\nu_n^\#)$ and $(b_0(\mu),b_1(\mu))=(0,0)$ or $\mu_0\in A_{n-1}$ and $(b_0(\mu),b_1(\mu))=(r_0,r_0)$. The number of $2$-$\Pop$-sortable elements of $\Tam(\nu_n^\#)$ is $g(n)$, so  
\begin{equation}\label{EqMeet4}
|A_n|=g(n)+|A_{n-1}|. 
\end{equation}
The only element of $A_1$ is $\nu_1$ itself, which has $\nu_1$-bracket vector $(0,0,\ldots,0,1,1,\ldots,1)$, so $|A_1|=1$. 

There are $5$ choices for $\mu$ when we are in Case~1, which occurs if and only if $\mu_0\in A_{n-1}$. There are $3$ choices for $\mu$ when we are in Case~3, which occurs if and only if $\mu_0\in A_{n-1}'$. Finally, there are $2$ choices for $\mu$ when we are in Case~5, which occurs if and only if $\mu_0$ is a $2$-$\Pop$-sortable element of $\Tam(\nu_n^\#)\setminus(A_{n-1}\cup A_{n-1}')$. Therefore, $g(n+1)=5|A_{n-1}|+3|A_{n-1}'|+2(g(n)-|A_{n-1}|-|A_{n-1}'|)=2g(n)+3|A_{n-1}|+|A_{n-1}'|$. Since $|A_{n-1}'|=1$, we have  
\begin{equation}\label{EqMeet5}
g(n+1)=2g(n)+3|A_{n-1}|+1.
\end{equation} 
Observe that the equations \eqref{EqMeet4} and \eqref{EqMeet5} still hold for $n=0$ and $n=1$ if we set $g(0)=1$ and $A_{-1}=A_0=\emptyset$. Put $G(z)=\sum_{n\geq 1}g(n)z^n$ and $R(z)=\sum_{n\geq 1}|A_{n-1}|z^n$. The equation \eqref{EqMeet4} translates into the generating function equation $R(z)=zG(z)+zR(z)$. The equation \eqref{EqMeet5} translates into $G(z)=2zG(z)+3zR(z)+\frac{z}{1-z}$. Solving this system of two equations yields $G(z)=\frac{z}{1-3z-z^2}$. 
\end{proof}

\section{Future Directions}\label{SecConclusion}

We have investigated the semilattice pop-stack-sorting operators $\Pop_M$, concentrating mostly on $\nu$-Tamari lattices. In the article \cite{DefantCoxeterPop}, the author explored these operators on weak orders of Coxeter groups. It would be nice to have other examples of meet-semilattices for which the semilattice pop-stack-sorting operators have especially structured yet nontrivial behavior. For example, it could be worth investigating these operators on the lattices $M_{\bf c}$ defined via chip-firing in the introduction.  

Recall that if $M$ is a meet-semilattice, then a map $f:M\to M$ is called \emph{compulsive} if $f(x)\leq\Pop_M(x)$ for all $x\in M$. It could be interesting to investigate which meet-semilattices $M$ have the property that $\sup\limits_{x\in M}\left|O_f(x)\right|\leq\sup\limits_{x\in M}\left|O_{\Pop_M}(x)\right|$ for every compulsive map $f:M\to M$. Indeed, the author proved in \cite{DefantCoxeterPop} that weak orders of irreducible Coxeter groups have this property, and he exhibited a $6$-element lattice that does not have this property. 

Theorem~\ref{ThmMeet4} states that the $m$-ballot paths in $\Tam_n(m)$ whose forward orbits under $\Pop_{\Tam_n(m)}$ are of maximum size are counted by the same numbers that count primitive $m$-ballot paths. It would be interesting to have a more direct explanation of this fact. 

In Sections~\ref{Subsec1Sortable} and~\ref{Subsec2Sortable}, we enumerated $1$-$\Pop$-sortable and $2$-$\Pop$-sortable $m$-ballot paths. Motivated by these results, we state the following conjecture. 

\begin{conjecture}\label{ConjMeet1}
Fix integers $m,t\geq 1$, and let $h_t(m,n)$ be the number of $t$-$\Pop$-sortable $m$-ballot paths in $\Tam_n(m)$. The generating function \[\sum_{n\geq 1}h_t(m,n)z^n\] is rational. 
\end{conjecture}

It would be interesting to have a proof of Conjecture~\ref{ConjMeet1} even just for Tamari lattices (i.e., the case when $m=1$). 

Finally, let us recall Problem~\ref{ProbMeet1}, which asks for a characterization of finite $\Pop$-trivial lattices.

\section{Acknowledgments}
The author was supported by a Fannie and John Hertz Foundation Fellowship and an NSF Graduate Research Fellowship (grant number DGE 1656466). The author thanks Letong Hong, Henri M\"uhle, Nathan Williams, and the anonymous referee for providing helpful feedback on this paper.

\end{document}